\documentclass[12pt]{amsart}
\usepackage{amssymb, amsmath}
\usepackage{stmaryrd}

\usepackage[dvipsnames]{xcolor}
\usepackage{color}
\usepackage[utf8]{inputenc}
\usepackage[T1]{fontenc}
\usepackage{fullpage}
\usepackage[normalem]{ulem}
\usepackage{amsmath,amscd,mathrsfs}
\usepackage{tikz}
\usepackage[colorlinks=true, citecolor=blue, linkcolor=red,pagebackref, hyperindex]{hyperref}
\usepackage[all,cmtip]{xy}

\newtheorem{Theorem}{Theorem}[section]

\newtheorem{Potential Theorem}[Theorem]{Potential Theorem}
\newtheorem{Lemma}[Theorem]{Lemma}
\newtheorem{Corollary}[Theorem]{Corollary}
\newtheorem{Proposition}[Theorem]{Proposition}

\newtheorem*{Claim*}{Claim}
\theoremstyle{definition}
\newtheorem{Example}[Theorem]{Example}

\newtheorem{Conjecture}[Theorem]{Conjecture}
\newtheorem{Definition}[Theorem]{Definition}
\newtheorem{Question}[Theorem]{Question}

\newtheorem{Discussion}[Theorem]{Discussion}
\theoremstyle{remark}
\newtheorem{Remark}[Theorem]{Remark}

\DeclareMathOperator{\depth}{depth}
\DeclareMathOperator{\height}{ht}

\DeclareMathOperator{\Max}{Max}

\DeclareMathOperator{\id}{id}
\DeclareMathOperator{\Proj}{Proj}

\DeclareMathOperator{\Hom}{Hom}
\DeclareMathOperator{\Spec}{Spec}

\DeclareMathOperator{\Ass}{Ass}
\DeclareMathOperator{\Tor}{Tor}
\DeclareMathOperator{\Ext}{Ext}
\DeclareMathOperator{\pd}{pd}

\DeclareMathOperator{\rank}{rank}

\DeclareMathOperator{\reg}{reg}

\DeclareMathOperator{\cd}{cd}

\DeclareMathOperator{\red}{red}

\def\char{\mbox{char}\,}

\def\m{\mathfrak{m}}
\def\n{\mathfrak{n}}
\def\a{\mathfrak{a}}

\def\Z{\mathbb{Z}}

\def\D{\mathscr{D}}

\def\N{\mathbb{N}}

\def\CC{\mathbb{C}}

\def\ds{\displaystyle}

\renewcommand{\geq}{\geqslant}
\renewcommand{\leq}{\leqslant}

\newcommand{\ck}[1]{{#1}^{\vee}}

\newcommand{\ps}[1]{\llbracket {#1} \rrbracket}

\newcommand{\lr}[1]{\left\langle{#1}\right\rangle}

\usepackage{comment}

\begin{document}

\title{Cohomologically full rings}
\author{Hailong Dao}
\address{Department of Mathematics, University of Kansas, 405 Snow Hall, 1460
Jayhawk Blvd., Lawrence, KS 66045}
\email{hdao@ku.edu}
\author{Alessandro De Stefani}
\address{Department of Mathematics, University of Nebraska, 203 Avery Hall, Lincoln, NE 68588}
\email{adestefani2@unl.edu}
\author{Linquan Ma}
\address{Department of Mathematics, University of Utah, Salt Lake City, Utah 84112}
\email{lquanma@math.utah.edu}

\thanks{The first author is partially supported by  NSA Grant H98230-16-1-0012. The third named author was supported in part by NSF Grant \#1836867/1600198 and NSF CAREER Grant DMS \#1252860/1501102.}

\begin{abstract}
Inspired by a question raised by Eisenbud-Musta\c{t}\u{a}-Stillman regarding the injectivity of maps from $\Ext$ modules to local cohomology modules and the work by the third author with Pham, we introduce a class of rings which we call cohomologically full rings. This class of rings includes many well-known singularities: Cohen-Macaulay rings, Stanley-Reisner rings,  F-pure rings in positive characteristics, Du Bois singularities in characteristics $0$. We prove many basic properties of cohomologically full rings, including their behavior under flat base change. We show that ideals defining these rings satisfy many desirable properties, in particular they have small cohomological and projective dimension. When $R$ is a standard graded algebra over a field of characteristic $0$, we show under certain conditions that being cohomologically full is equivalent to the intermediate local cohomology modules being generated in degree $0$.  Furthermore, we obtain Kodaira-type vanishing and strong  bounds on the regularity of cohomologically full graded algebras.
\end{abstract}

\maketitle

\section{Introduction}
Throughout this paper, all rings we consider are commutative, Noetherian, and contain an identity element. For an ideal $I$ of a ring $S$, the local cohomology modules $H_I^i(S)$ can be described as $H_I^i(S)=\varinjlim_e\Ext_S^i(S/I_e, S)$ for all $i\geq 0$, where $\{I_e\}$ is a decreasing sequence of ideals cofinal with the ordinary powers $\{I^e\}$, and the maps in the directed system are induced by the natural surjections. Clearly, it would be hugely beneficial if the direct limit above is actually a {\it union}, for then many questions about local cohomology can be reduced to understanding finitely generated modules. Motivated by such idea, when $S$ is a polynomial ring over a field, Eisenbud-Musta\c{t}\u{a}-Stillman \cite[Question 6.2]{EisenbudMustataStillman} asked for what ideals $I$ the natural map $\Ext^i_S(S/I, S)\to H_I^i(S)$ is an injection. Since then, there have been many partial results towards answering this question, for example see \cite{Mustata_monomial,SinghWalther,MaSchwedeShimomoto}.

We observe that, quite generally for any regular local ring $(S,\m)$ of dimension $n$, the natural map $\Ext^{n-i}_S(S/I, S)\to H_I^{n-i}(S)$ is an injection provided $H_\m^i(S/J)\to H_\m^i(S/I)$ is a surjection for all ideals $J$ that satisfy $J\subseteq I\subseteq\sqrt{J}$, by local duality (see Proposition \ref{Prop_characterization_cohom_full} for more details). Motivated by this, we introduce the following  intrinsic definition, which will be the main object of study of this article.

\begin{Definition}
\label{Defn_full}
Let $(R,\m,k)$ be a local ring, and let $R_{\red} = R/\sqrt{0}$. We say that $R$ is {\it $i$-cohomologically full} if, for every surjection $(T,\n) \twoheadrightarrow R$ such that $T$ and $R$ have the same characteristic\footnote{This means that either both $T$ and $R$ contain a field, or both $T$ and $R$ have mixed characteristic, i.e., the characteristic of the ring differs from the characteristic of its residue field.} and $T_{\red} = R_{\red}$, the natural map
\[
\ds H^i_\n(T) \to H^i_\m(R)
\]
is surjective. We say that $R$ is {\it cohomologically full} if it is $i$-cohomologically full for all $i$. 
\end{Definition}

The definition can be adapted easily to the (standard) graded case, see Definition \ref{Defn_fullgraded}.  In this paper we investigate the ubiquity and properties of cohomologically full rings.  Our main findings can be summarized below.

\begin{enumerate}
\item Under very mild assumptions, ideals whose quotients are cohomologically full  are precisely the ones that satisfy Eisenbud-Musta\c{t}\u{a}-Stillman's question mentioned above. In positive characteristics, they also answer completely the original motivational question: $S/I$ is cohomologically full if and only if the local cohomology module $H_I^i(S)$ can be written as a directed union of the $\Ext$ modules. See \ref{Prop_characterization_cohom_full} and \ref{EMS_Q1}.

\item The class of cohomologically full rings strictly includes many classes of well-known singularities: Cohen-Macaulay local rings, F-pure rings in positive characteristics, Du Bois singularities in characteristics $0$.  In fact, we show that when $R$ has characteristic $p>0$, cohomologically full rings are exactly F-full rings introduced in \cite{MaQuy}. We extend tight bounds on cohomological dimension and projective dimension of ideals defining these singularities to cohomologically full rings. See \ref{Cor_Cohfull=Ffull}, \ref{Remark_cohomfull}, \ref{Prop_depth} and \ref{mu}.

\item We give a simple  characterization of cohomological fullness when $R$ is an equidimensional standard graded algebra over a field of characteristic $0$ and $\Proj R$ has only Cohen-Macaulay, Du Bois singularities. See Theorem \ref{H0generates}. We discuss the connection to and a potential strengthening of the weak ordinarity conjecture by Musta\c{t}\u{a}-Srinivas in \cite{MustataSrinivasOrdinaryVarieties}, see Section 5.

\item We study carefully how cohomological fullness behaves under some basic operations: gluing, base change, reduction to positive characteristics. These properties allow us to classify this class of rings in small dimensions and give a wealth of examples. See  \ref{glue}, \ref{Prop_connected_components_pctSpec}, \ref{coh_full_dim2} and Section 3.

\item We establish a strong bound on regularity of a homogeneous ideal $I$ in a polynomial ring $S$ such as $S/I$ is cohomologically full. Roughly speaking, the bound is just the number of generators of $I$ times the maximal degree of the generators. Intriguingly, our proof uses reduction to characteristic $p$. See Subsection \ref{regularity}.

\item Cohomologically full standard graded $k$-algebras that are Cohen-Macaulay on the punctured spectrum satisfy Kodaira-type vanishing theorem and certain Lyubeznik numbers can be read off as the $0$-pieces of local cohomology modules. Such statements partially generalize previous results on the singularities mentioned above. See Subsections \ref{Kod} and \ref{Lyu}.
\end{enumerate}

Beyond these results, there is more evidence that cohomologically full rings satisfy many desirable properties. In fact, our definition of cohomologically fullness coincides with the concept of {\it having liftable local cohomology} over certain base recently and independently introduced by Koll\'ar-Kov\'acs \cite{KollarKovacsDeformationsoflogcanonicalFpure}, who obtained interesting and important results on the base change of the cohomologies of the relative dualizing complexes. Furthermore, in a recent preprint \cite{ConcaVarbaroSquarefreeGroebner}, Conca and Vabaro used the ideas of cohomologically full rings together with the results in \cite{KollarKovacsDeformationsoflogcanonicalFpure} to settle a conjecture by Herzog on ideals with square-free initial ideals.

This paper is organized as follows. In Section 2 we prove some basic properties of cohomologically full rings and in Section 3 we study the behavior of cohomologically full rings under various base change, properties (1),(2),(4) mentioned above will be proved in these sections. In Section 4 and Section 5, we prove the aforementioned (3),(5),(6) as well as some other applications. Examples will be given throughout.

\vspace{1em}

\noindent\textbf{Acknowledgements}: We would like to thank S\'andor Kov\'acs for pointing out an error in an earlier version of the paper. We thank David Eisenbud, Mircea Musta\c{t}\u{a}, Pham Hung Quy, Shunsuke Takagi and Matteo Varbaro for some useful discussions and comments, and Luis N{\'u}{\~n}ez-Betancourt and Ilya Smirnov for sharing Proposition \ref{Prop_Lyubeznik R/x} with them.

\section{Basic properties}

In this section we prove some basic properties of cohomologically full rings. We begin with the following result which gives alternative characterizations of cohomologically full rings. For example we will see that, under very mild assumptions, cohomologically full rings are {\it precisely} those for which Eisenbud-Musta\c{t}\u{a}-Stillman's question has a positive answer. This will be our main tool for studying properties of this class of rings.

We recall that a regular local ring $(S,\m)$ is unramified if either $S$ has equal characteristic, or $S$ has mixed characteristic $(0,p)$ and $p\notin\m^2$. Equivalently, the $\m$-adic completion of $S$ is either a power series ring over a field, or a power series ring over a complete and unramified discrete valuation ring of mixed characteristic.

\begin{Proposition}\label{Prop_characterization_cohom_full} Let $(R,\m,k)$ be a local ring. Consider the following conditions:
\begin{enumerate}
\item $R$ is $i$-cohomologically full.
\item For every surjection $A \twoheadrightarrow R \cong A/I$ from a regular local ring $A$ such that $A$ and $R$ have the same characteristic, and every ideal $J \subseteq I$ with $\sqrt{J}=\sqrt{I}$, the natural map $H^i_\m(A/J) \to H^i_\m(R)$ is surjective.
\item For every surjection $A \twoheadrightarrow R\cong A/I$ from a regular local ring $A$ such that $A$ and $R$ have the same characteristic, and every sequence of ideals $\{I_e\}$ of $A$, cofinal with the ordinary powers $\{I^e\}$, the natural map $H^i_\m(A/I_e) \to H^i_\m(R)$ is surjective for all $e$.
\item For every surjection $A \twoheadrightarrow R\cong A/I$ from a regular local ring $A$ such that $A$ and $R$ have the same characteristic, the natural map $\Ext^{n-i}_A(R,A) \to H^{n-i}_I(A)$ is injective, where $n=\dim(A)$.
\end{enumerate}
Then we have $(1)\Rightarrow(2)\Rightarrow(3)\Rightarrow(4)$.

Moreover, if $R$ is a homomorphic image of an unramified regular local ring $S$ such that $S$ and $R$ have the same characteristic, say $R=S/I$, then $(1)-(4)$ are all equivalent to:
\begin{enumerate}
\setcounter{enumi}{4}
\item The natural map $\Ext^{n-i}_S(R,S) \to H^{n-i}_I(S)$ is injective, where $n=\dim(S)$.
\item If $\{I_e\}$ is any sequence of ideals in $S$ cofinal with the ordinary powers $\{I^e\}$, the natural map $H^i_\m(S/I_e) \to H^i_\m(R)$ is surjective for all $e$.
\end{enumerate}
\end{Proposition}
\begin{proof}
$(1)\Rightarrow(2)\Rightarrow(3)$ are clear. $(4)$ follows from $(3)$ after applying Matlis duality to the surjections $H^i_\m(A/I_e) \to H^i_\m(R)$, and taking the direct limit over $e$. The implication $(4) \Rightarrow (5)$ is also clear. If $\{I_e\}$ is any sequence of ideals in $S$ cofinal with $\{I^e\}$, then $\Ext^{n-i}_S(R,S) \to H^{n-i}_I(S)$ is injective if and only if $\Ext^{n-i}_S(R,S) \to \Ext^{n-i}_S(S/I_e,S)$ is injective for all $e$, since $H^{n-i}_I(S)$ is the direct limit over $e$ of the modules $\Ext^{n-i}_S(S/I_e,S)$. By Matlis duality, this is equivalent to $H^i_\m(S/I_e) \to H^i_\m(R)$ being surjective for all $e$, proving that $(5)$ and $(6)$ are equivalent.

Now, assuming that $R=S/I$ is a homomorphic image of an unramified regular local ring $(S,\m,k)$ such that $S$ and $R$ have the same characteristic, we will show $(5)\Rightarrow(1)$, and this will complete the proof. Let $T \twoheadrightarrow R$ be a surjective ring homomorphism with $T_{\red}=R_{\red}$ and such that $T$ and $R$ have the same characteristic. We want to show that the induced map $H^i_\n(T) \to H^i_\m(R)$ is surjective. Since passing to $\m$-adic completions does not affect the modules $H_\n^i(T)$ and $H_\m^i(R)$, and it does not affect whether $\Ext^{n-i}_S(R,S) \to H^{n-i}_I(S)$ is injective or not, we may assume without loss of generality that $R, S, T$ are all complete local rings. We pick a coefficient ring for $V$ for $T$: in equal characteristic $V\cong k$ is a field; in mixed characteristic, $(V,pV)$ is a complete and unramified discrete valuation ring with $V/pV\cong k$. We have
\[\xymatrix{
& & S\ar@{->>}[d] \\
V\ar[r] &  T \ar[r]  & R.
}
\]
By \cite[(19.8.6.(i))]{EGAIVI}, there exists a map $V\to S$ making the diagram commute. Therefore $V$ can be viewed as a coefficient ring for $S$ since $R,S,T$ have the same characteristic and the same residue field. Since $S$ is complete and unramified, and $V$ is a coefficient ring for $S$, by Cohen's structure theorem we have $S\cong V\ps{\underline{x}}$, where $\underline{x}$ denotes a set of $n$ or $n-1$ variables over $V$, depending on whether $S$ has equal characteristic or mixed characteristic. We have a commutative diagram
\[\xymatrix{
V\ar[r]\ar[d] & S\cong V\ps{\underline{x}} \ar@{->>}[d] \\
 T \ar[r]  & R
}
\]
Let $J$ be the kernel of $T\twoheadrightarrow R$. Since $S$ is formally smooth over $V$ and $J$ is nilpotent in $T$ (since $T_{\red}=R_{\red}$), there is a map $S\to T$ making the above diagram commutes. If we further let $y_1,\ldots,y_m$ be elements in $J$ that form a basis for $\frac{J+\n^2}{\n^2}$, we have a commutative diagram:
\[\xymatrix{
S_0=S\ps{y_1,\dots,y_m}\ar[r]\ar@{->>}[d] & S \ar@{->>}[d] \\
 T \ar[r]  & R,
}
\]
where the map on the first line is the natural map sending $y_i$ to $0$. We write $T=S_0/J_0$ and $R=S_0/I_0$ with $J_0\subseteq I_0$ and $\sqrt{J_0}=\sqrt{I_0}$. Since $\dim S_0=m+n$, by local duality, in order to prove that $H^i_\n(T) \to H^i_\m(R)$ is surjective, it is enough to show the map $\Ext^{m+n-i}_{S_0}(S_0/I_0,S_0)\rightarrow \Ext^{m+n-i}_{S_0}(S_0/J_0,S_0)$ is injective. This follows if we can show that the natural map $\Ext^{m+n-i}_{S_0}(S_0/I_0,S_0)\rightarrow H_{I_0}^{m+n-i}(S_0)$ is injective. At this point we note that $I_0=I+(y_1,\dots,y_m)$, where $I$ is actually an ideal of $S$. Therefore, by induction on $m$, it is enough to prove that condition $(5)$ implies that the map $\Ext^{n+1-i}_{S\ps{y}}(S\ps{y}/(I+y), S\ps{y})\rightarrow H_{I+y}^{n+1-i}(S\ps{y})$ is injective. We look at the exact sequence:
\[
\xymatrix{
\ds 0\ar[r] & \Ext^{n-i}_S(S/I, S)\ar[r] & H_I^{n-i}(S) \ar[r] &  C\ar[r] & 0,
}
\]
where the first map is injective by assumption. Applying the functors $- \otimes_S S\ps{y}$ first and $\Gamma_y(-)$ after that, we get an exact sequence:
\[
\xymatrixcolsep{4mm}
\xymatrix{
H_y^0(C\otimes_S S\ps{y})\ar[r] & H_y^1(\Ext^{n-i}_{S\ps{y}}(S\ps{y}/I, S\ps{y}))\ar[r] & H_y^1(H_I^{n-i}(S\ps{y})) \ar[r] & H_y^1(C\otimes_S S\ps{y}).
}
\]
Since $y$ is a nonzerodivisor on $C\otimes_S S\ps{y}$, we see that $H_y^0(C\otimes_S S\ps{y})=0$. Moreover, since $y$ is a nonzerodivisor on $\Ext^{n-i}_{S\ps{y}}(S\ps{y}/I, S\ps{y})$ it is easy to see that
\[
\ds \Ext^{n-i+1}_{S\ps{y}}(S\ps{y}/(I+y), S\ps{y})\cong \frac{\Ext^{n-i}_{S\ps{y}}(S\ps{y}/I, S\ps{y})}{y\Ext^{n-i}_{S\ps{y}}(S\ps{y}/I, S\ps{y})} \hookrightarrow H_y^1(\Ext^{n-i}_{S\ps{y}}(S\ps{y}/I, S\ps{y})).
\]
Finally, a spectral sequence for local cohomology gives
\[
\ds H_y^1(H_I^{n-i}(S\ps{y}))\cong H_{I+y}^{n+1-i}(S\ps{y}),
\]
and combining all the above we finally obtain injectivity of the map
\[
\ds \Ext^{n+1-i}_{S\ps{y}}(S\ps{y}/(I+y), S\ps{y})\hookrightarrow H_{I+y}^{n+1-i}(S\ps{y}).
\]
This finishes the proof.
\end{proof}

Recall that a local ring $(R,\m, k)$ of characteristic $p>0$ is called {\it F-full} if the natural map $\mathcal{F}_R(H^i_\m(R)) \to H^i_\m(R)$ is surjective for all integers $i$, where $\mathcal{F}_R$ denotes the Peskine-Szpiro base change functor of the Frobenius.
\begin{Corollary}
\label{Cor_Cohfull=Ffull}
Suppose $R$ is a homomorphic image of a regular local ring of characteristic $p>0$. Then $R$ is cohomologically full if and only if $R$ is F-full. \end{Corollary}
\begin{proof}
We write $R=S/I$ for a regular local ring $S$ of characteristic $p>0$. Setting $I_e = I^{[p^e]}$ and using the equivalence $(1) \Leftrightarrow (6)$ in Proposition \ref{Prop_characterization_cohom_full}, we obtain that $R$ is cohomologically full if and only if $H_\m^i(S/I^{[p^e]})\to H_\m^i(S/I)$ is surjective for all $i$ and all $e$. But the image of this map is the same as the image of $\mathcal{F}^e_R(H_\m^i(R))\to H_\m^i(R)$ by \cite[Lemma 2.2]{Lyubeznik_vanishing}. Therefore $R$ is cohomologically full if and only if $R$ is F-full.
\end{proof}

So far our Definition \ref{Defn_full} and Proposition \ref{Prop_characterization_cohom_full} are restricted to local rings. But they can be easily adapted to the graded set up.

\begin{Definition}
\label{Defn_fullgraded}
Let $(R,\m,k)$ be a standard graded $k$-algebra (i.e., $R$ can be generated by finitely many elements of degree one over $k$ and $\m$ denote the irrelevant maximal ideal). We say $R$ is {\it cohomologically full} if $R_{\m}$ is cohomologically full.
\end{Definition}

If we write $R=S/I$ where $S=k[x_1,\ldots,x_n]$ is a standard graded polynomial ring, then by Definition \ref{Defn_fullgraded} and Proposition \ref{Prop_characterization_cohom_full}, $R$ is cohomologically full if and only if for any sequence of homogeneous ideals $\{I_e\}$ in $S$ cofinal with the ordinary powers $\{I^e\}$, the natural map $H^i_\m(S/I_e) \to H^i_\m(R)$ is surjective for all $i$ and $e$. Moreover, by graded local duality, this holds if and only if the natural map $\Ext^{n-i}_S(R,S) \to H^{n-i}_I(S)$ is injective for all $i$, where $n=\dim(S)$. Therefore the analog of Proposition \ref{Prop_characterization_cohom_full} holds in this setup too.




\begin{Remark}
\label{Remark_cohomfull} We collect some immediate consequences of Proposition \ref{Prop_characterization_cohom_full} and Corollary \ref{Cor_Cohfull=Ffull}.
\begin{enumerate}
 \item The proof of Proposition \ref{Prop_characterization_cohom_full} $(5)\Rightarrow(1)$ actually shows that, when $R$ is a homomorphic image of an unramified regular local ring $S$ such that $S$ and $R$ have the same characteristic, $R$ is $i$-cohomologically full if and only if $\widehat{R}$ is $i$-cohomologically full.
 \item Since F-pure local rings are F-full by \cite[Remark 2.4]{MaQuy} (see also \cite[Theorem 1.1]{MaFinitnessFpure}), by Corollary \ref{Cor_Cohfull=Ffull} F-pure local rings are cohomologically full.
 \item Suppose $(R,\m,k)$ is a reduced local ring essentially of finite type over $\mathbb{C}$. If $R$ has Du Bois singularities, then $R$ is cohomologically full by \cite[Lemma 3.3]{MaSchwedeShimomoto}.
  \item If $(R,\m,k)$ is a local ring of dimension $d$, then $R$ is always $d$-cohomologically full. In particular, Cohen-Macaulay rings are always cohomologically full.
  \item Stanley-Reisner rings (i.e., quotient of polynomial or power series rings over a field by square-free monomials) are cohomologically full. This follows from $(2)$ and $(3)$ above, and also directly from \cite[Theorem 1 (i)]{Lyubeznik_monomial} and \cite[Theorem 1.1]{Mustata_monomial}.
\end{enumerate}
\end{Remark}

Recall that, given a proper ideal $I$ inside a ring $S$, the {\it cohomological dimension of $I$ in $S$} is defined to be $\cd(I,S) = \sup\{j \in \N \mid H^j_I(S) \ne 0\}$. We always have inequalities $\height(I) \leq \cd(I,S) \leq \mu(I)$, where $\mu(-)$ denotes the {\it minimal number of generators} of a module.

\begin{Proposition} \label{Prop_depth} Let $(S,\m,k)$ be a regular local ring of dimension $n$, and $I\subseteq S$ be an ideal. Suppose $S/I$ is cohomologically full and that  $S$ and $S/I$ have the same characteristic. Then $\depth(S/I) \geq n-\cd(I,S)$. Furthermore, equality holds if any of the following additional conditions is satisfied:
\begin{enumerate}
\item There exists a sequence of ideals $\{I_e\}$, cofinal with $\{I^e\}$, such that $\depth(S/I) = \depth(S/I_e)$ for all $e$. In particular, the equality holds when $\char(S)=p>0$.
\item $\depth(S/I)\leq 2$.
\item $S$ is essentially of finite type over a field of characteristic $0$ and $\depth(S/I) \leq 3$.
\item $S$ is essentially of finite type over a field of characteristic $0$, $\depth(S/I) = 4$, and the local Picard group of the completion $\widehat{S/I}$ is torsion.
\end{enumerate}
\end{Proposition}
\begin{proof}
By Proposition \ref{Prop_characterization_cohom_full} $(1)\Rightarrow(4)$ we conclude that $\Ext^j_S(S/I,S) \hookrightarrow H^j_I(S)$ is an injection for all $j$. It follows by local duality that
\[
\ds \depth(S/I) = n-\max\{j \mid \Ext^j_S(S/I,S)\ne 0\} \geq n-\cd(I,S).
\]
The reverse inequality is well-known to hold when $\depth(S/I) \leq 2$ (see for example \cite[Proposition 3.1]{Varbaro}). When $\depth(S/I)=3$ or when $\depth(S/I) = 4$ and the local Picard group of the completion $\widehat{S/I}$ is torsion, this follows from \cite[Corollary 2.8 and Theorem 2.9]{DaoTakagi}. If there is a sequence of ideals $\{I_e\}$ that is cofinal with the ordinary powers $\{I^e\}$ such that $\depth(S/I_e) = \depth(S/I)$, then we obtain that $\Ext^j_S(S/I,S) = 0$ if and only if $H^j_I(S) =\varinjlim_{e} \Ext^j_S(S/I_e,S)= 0$ (and when $\char(S)=p>0$ we can take $I_e=I^{[p^e]}$).
\end{proof}

\begin{Corollary} \label{Coroll_Positive_Depth} Let $(S,\m,k)$ be a regular local ring of dimension $n$, and $I\subseteq S$ be an ideal such that $S$ and $S/I$ have the same characteristic. If $S/I$ has positive dimension, and is cohomologically full, then $\depth(S/I)>0$. In particular, if $S/I$ is one dimensional, then it is cohomologically full if and only if it is Cohen-Macaulay.
\end{Corollary}
\begin{proof}
The lower bound for the depth is immediate from Proposition \ref{Prop_depth} and the Hartshorne-Lichtenbaum Vanishing Theorem \cite{HartshorneLichtenbaum}. The second claim follows from Remark \ref{Remark_cohomfull} (4).
\end{proof}

Proposition \ref{Prop_depth} also recovers upper bounds to the projective dimension of a cohomologically full ring that are analogous to those obtained for F-pure or F-injective rings in characteristic $p>0$ and for Du Bois rings in characteristic $0$. For more details, see for example \cite{SinghWalther,DaoHunekeSchweig,DSNB_F-thresholds_graded,MaSchwedeShimomoto}.

\begin{Corollary} \label{mu} Let $(S,\m,k)$ be a regular local ring, and $I\subseteq S$ be an ideal such that $S$ and $S/I$ have the same characteristic. If $S/I$ is cohomologically full, then $\pd_S(S/I) \leq \mu(I)$.
\end{Corollary}
\begin{proof}
Let $n=\dim(S)$. It is sufficient to observe that $\pd_S(S/I) = n-\depth(S/I) \leq \cd(I,S) \leq \mu(I)$, where the only nontrivial inequality follows from Proposition \ref{Prop_depth}.
\end{proof}

We next present some general ``gluing results'' for cohomologically full rings.

\begin{Proposition}\label{glue}
Let $(S,\m)$ be an unramified regular local ring, and $J,K$ be two ideals of $S$ such that $S$, $S/J$, and $S/K$ have the same characteristic. Let $R=S/I$, where $I=J \cap K$. Consider the following integers:
\[
l= \max\{\pd_S(S/J), \pd_S(S/K), \pd_S(S/I) \}, \ \ l'=\max\{\pd_S(S/J), \pd_S(S/K)\}, \ \ h= \height(J+K).
\]
\begin{enumerate}
\item Suppose that $l< h$. Then $R$ is cohomologically full if  and only if both $S/J$ and $S/K$ are cohomologically full.
\item Suppose that $l'<h$ and $S/J+K$ is cohomologically full. Then  $R$ is cohomologically full  if and only if  $S/J$, $S/K$ are cohomologically full.
\item Suppose that $l'<h$ and there exists a sequence of ideals $\{I_e\}$, cofinal with the powers $\{I^e\}$, such that $\depth(S/I)=\depth(S/I_e)$ for all $e$. Then  $R$ is cohomologically full  if and only if   $S/J$, $S/K$ and $S/J+K$ are cohomologically full.
 \end{enumerate}

\end{Proposition}
\begin{proof}
Without loss of generality, we can assume that $S$ is complete. For all integers $j$, we have the following commutative diagram:
\[\xymatrix{
\Ext^{j}_S\left(\frac{S}{J+K}, S\right) \ar[r] \ar[d] & \Ext^j_S\left(\frac{S}{J}, S\right)\oplus \Ext^j_S(\frac{S}{K}, S) \ar[r] \ar[d] &  \Ext_S^j(\frac{S}{I}, S) \ar[r] \ar[d] & \Ext^{j+1}_S\left(\frac{S}{J+K}, S\right) \ar[d] \\
H_{J+K}^{j}(S) \ar[r] &  H_J^j(S)\oplus H_K^j(S) \ar[r] & H_I^j(S) \ar[r] & H_{J+K}^{j+1}(S)
}
\]
We first prove $(1)$. To this end, we can assume that $j\leq l$ in the diagram above. But then $\Ext^{j}_S(S/J+K, S)= H_{J+K}^{j}(S)=0$, as $j<h$. We also have $\Ext^{j+1}_S(S/J+K, S)\hookrightarrow H_{J+K}^{j+1}(S)$, since $j+1\leq h$. Chasing the above diagram immediately shows that the injectivity of the two middle vertical maps are equivalent.

The case of $(2)$ is similar for $j\leq l'$. For $j>l'$, then $\Ext^j_S(S/J, S)\oplus \Ext^j_S(S/K, S)=0$, so again we have that the vertical  map  $ \Ext_S^j(S/I, S) \to H^j_I(S)$ must be injective since the one on the right is.

Finally, we prove $(3)$. Given part $(2)$, the only claim left to show is that, if $R$ is cohomologically full, then so is $S/J+K$. Assume that $R$ is cohomologically full. Since $S/I$ and $S/K$ are also cohomologically full, it follows from Proposition \ref{Prop_depth} that $\cd(J,S) = \pd_S(S/J)$ and $\cd(K,S) = \pd_S(S/K)$, so that $\max\{\cd(J,S),\cd(K,S)\}\leq l'$. For each $j\geq l'+1$, the tail of the diagram above breaks into squares:
\[\xymatrix{
  \Ext_S^j(S/I, S) \ar[r] \ar[d] & \Ext^{j+1}_S(S/J+K, S) \ar[d] \\
 H_I^j(S) \ar[r] & H_{J+K}^{j+1}(S)
}
\]
where the horizontal maps are isomorphisms. Since $S/I$ is assumed to be cohomologically full, it follows that $S/J+K$ must be cohomologically full as well.
\end{proof}

\begin{Remark}
In Proposition \ref{glue}, the assumptions regarding $l$ and $l'$ in relation to $h$ cannot be relaxed, even when $J+K$ is $\m$-primary. For example, let $(S,\m)$ be a regular local ring, $J$ be a prime ideal of dimension $1$, and $K=\m^c$, with $J \not\subseteq K$. Then $\depth(S/J\cap K)=0$, so $S/J\cap K$ cannot be cohomologically full by Corollary \ref{Coroll_Positive_Depth}, even if $S/J$ and $S/K$ are. Note that Proposition \ref{glue} does not apply, since $l = l' = h$ in this case.
\end{Remark}

\begin{Proposition} \label{Prop_connected_components_pctSpec}
Let $(S,\m)$ be an unramified regular local ring, and $I$ be an ideal of $S$ such that $S$ and $S/I$ have the same characteristic. Assume that $R=S/I$ has positive depth, and write $I= I_1\cap I_2 \cap \dots \cap I_r$, where each $I_i$ corresponds to a connected component of $\Spec R\smallsetminus\{\m\}$. Then $R$ is cohomologically full if and only if each $S/I_i$ is.
\end{Proposition}
\begin{proof}
If $\dim R =1$, being cohomologically is equivalent to being Cohen-Macaulay by Remark \ref{Remark_cohomfull} (4) and Corollary \ref{Coroll_Positive_Depth}. Since $R$ is assumed to have positive depth, the statement is then clear since each $S/I_i$ is also Cohen-Macaulay. We can therefore assume that $\dim R\geq 2$, and we induct on $r$, the number of connected components of $\Spec R\smallsetminus\{\m\}$. If $r=1$, the statement is a tautology. Suppose $r>1$, and let $J= I_1$ and $K= I_2\cap \dots \cap I_r$, so that $I= J\cap K$. Observe that $S/J+K$ is cohomologically full, since $J+K$ is $\m$-primary. Proposition \ref{glue} part (2) gives that $R$ is cohomologically full if and only if $S/J$ and $S/K$ are. The claim now follows by induction, since the punctured spectrum of $S/K$ has $r-1$ connected components.
\end{proof}

\begin{Corollary} \label{coh_full_dim2}
Let $(S,\m,k)$ be a complete unramified regular local ring with separably closed residue field $k$, and let $I$ be an ideal of $S$ such that $S$ and $S/I$ have the same characteristic. Assume that $R=S/I$ has dimension $2$. Let $I= I_1\cap I_2 \cap \dots \cap I_r$, where each $I_i$ corresponds to a connected components of $\Spec R\smallsetminus\{\m\}$. Then $R$ is cohomologically full if and only if $S/I_i$ is Cohen-Macaulay for each $i$.
\end{Corollary}
\begin{proof}
By Proposition \ref{Prop_connected_components_pctSpec}, we only need to show that if $\Spec R\smallsetminus \{\m\}$ is connected, then $R$ is Cohen-Macaulay. But the condition that $\Spec(R) \smallsetminus\{\m\}$ is connected implies that $\cd(S,I)\leq n-2$ \cite{Ogus,HartshorneSpeiser}, and it follows that $\depth R\geq 2$ by Proposition \ref{Prop_depth}.
\end{proof}

Let $J,K$ be two ideals in an unramified regular local ring $S$ such that $J\cap K= JK$ and such that both $S/J$ and $S/K$ are cohomologically full. One can ask if the rings $S/J\cap K$ and $S/(J+K)$ are cohomologically full as well. This is true when $S/J,S/K$ are Cohen-Macaulay. Because in this case we have $S/J+K$ is also Cohen-Macaulay by the depth formula \cite[Theorem 1.2]{Auslander}, and then we can use Proposition \ref{glue} for $S/J\cap K$. However, the answer is no in general as the following example shows.

\begin{Example} Let $S=k\ps{x,y,z}$ and $J=(xy,xz)$. Then $S/J$ is cohomologically full by Remark \ref{Remark_cohomfull} $(5)$. Let $r \in S$ be such that its image in $S/J$ is a nonzerodivisor. Then $S/(r)$ is cohomologically full (since it is Cohen-Macaulay) and $(r) \cap J = rJ$. However, $S/J+(r)$ cannot be cohomologically full for any choice of such $r$, since $\depth(S/J+(r))=0$ (see Corollary \ref{Coroll_Positive_Depth}).
\end{Example}

\section{Cohomologically full rings under base change}

In this section we study how cohomologically full rings behaves under various instances of base change.

\subsection{Deformation} We begin by proving the following deformation result for cohomologically fullness. This recovers the statement for F-full rings \cite[Theorem 4.2 (2)]{MaQuy}, and generalizes it to rings of any characteristic. We first recall that a nonzerodivisor $x$ in a local ring $(R, \m)$ is called a {\it surjective element} if the natural map on the local cohomology module $H_\m^i(R/(x^n))\to H_\m^i(R/(x))$ induced by $R/(x^n)\to R/(x)$ is surjective for all $n>0$ and $i\geq 0$. It is clear that if $R/(x)$ is cohomologically full, then $x$ is a surjective element.

\begin{Theorem}
\label{Thm_deformation}
Let $(R,\m,k)$ be local ring that is a homomorphic image of an unramified regular local ring $S$. Let $x \in \m$ be a nonzerodivisor on $R$. If $R/(x)$ is cohomologically full and that $S$, $R$, and $R/(x)$ have the same characteristic, then $R$ is cohomologically full.
\end{Theorem}
\begin{proof}
By Remark \ref{Remark_cohomfull} $(1)$ we may assume $R$ is complete. If $R/(x)$ is cohomologically full, then $H_\m^i(R/(x^k))\twoheadrightarrow H_\m^i(R/(x))$ for all $k>0$. In particular, $x$ is a surjective element. It follows from \cite[Proposition 3.3]{MaQuy} that the long exact sequence of local cohomology modules induced by
\[
\xymatrix{
0 \ar[r] & R \ar[r]^x & R\ar[r] & R/(x)\ar[r] &  0
}
\]
 splits into short exact sequences. That is, for every $i$, we have short exact sequences
\[
\xymatrix{
\ds 0\ar[r] & H_\m^{i-1}(R/(x))\ar[r] & H_\m^i(R)\ar[r]^-{x} & H_\m^i(R)\ar[r] & 0.
}
\]
Let $S\twoheadrightarrow R$ be a surjection where $S$ is a complete and unramified regular local ring such that $S$ and $R$ have the same characteristic. Taking the Matlis dual of the above sequence and applying local duality, we have the following commutative diagram
\[
\xymatrix{
0 \ar[r] & \Ext_S^j(R, S) \ar[r]^{x} \ar[d] & \Ext_S^j(R, S) \ar[r]^-\varphi \ar[d]^\alpha & \Ext_S^{j+1}(R/(x), S) \ar[r] \ar[d]^-\beta & 0\\
{\cdots} \ar[r] & H_I^j(S) \ar[r] & H_I^j(S)_x \ar[r] & H_{(I,x)}^{j+1}(S) \ar[r] & {\cdots}
}\]
where $j=\dim(S)-i$ and $\beta$ is injective by Proposition \ref{Prop_characterization_cohom_full}, because $R/(x)$ is cohomologically full. Now suppose there is $0 \ne \eta\in \Ext_S^j(R, S)$ such that $\alpha(\eta)=0$. We can write $\eta = x^n \eta'$ for some $\eta' \notin x \Ext^j_S(R,S)$. Since $\alpha(\eta) = 0$ by assumption, and $\alpha(\eta)$ and $\alpha(\eta')$ only differ by $x^n$ inside $H_I^j(S)_x$, we must have $\alpha(\eta') = 0$, as well. By commutativity of the above diagram, we have $\beta(\varphi(\eta')) = 0$ and, since $\beta$ is injective, we have $\varphi(\eta') = 0$. However, this means that $\eta' \in x \Ext^j_S(R,S)$, which contradicts the choice of $\eta'$. Therefore $\alpha$ is injective, and so is $\Ext^j_S(R,S) \to H_I^j(S)$, since $\alpha$ factors through this map. Thus, $R$ is cohomologically full by Proposition \ref{Prop_characterization_cohom_full}.
\end{proof}

\begin{Definition} Let $(R,\m,k)$ be a local ring, and $M$ be a finitely generated $R$-module. The finiteness dimension of $M$ with respect to $\m$ is defined as
\[
\ds f_\m(M) = \inf \{ t \in \Z \mid  H^t_\m(M) \mbox{ is not finitely generated} \}.
\]
\end{Definition}
Recall that $f_\m(M) < \infty$ if and only if $\dim(M)>0$, since $H^{\dim(M)}_\m(M)$ is infinitely generated in this case. The following result generalizes \cite[Remark 5.4]{MaQuy}.
\begin{Proposition}
Let $(R,\m,k)$ be a local ring, and $x \in \m$ be a nonzerodivisor. If $R/(x)$ is cohomologically full and that $R$ and $R/(x)$ have the same characteristic, then $\depth(R) = f_\m(R)$.
\end{Proposition}
\begin{proof}
Let $t=\depth(R)$. Clearly $\depth(R) \leq f_\m(R)$. Since $R/(x)$ is cohomologically full, $x$ is a surjective element. In particular, for all integers $n \geq 1$ we have an exact diagram
\[
\xymatrix{
0 \ar[r] & H^{t-1}_\m(R/(x^n)) \ar[d] \ar[rr] && H^t_\m(R) \ar[d]^-{\cdot x^{n-1}} \ar[rr]^{\cdot x^n} && H^t_\m(R) \ar@{=}[d] \ar[r] & 0\\
0 \ar[r] & H^{t-1}_\m(R/(x)) \ar[rr] && H^t_\m(R) \ar[rr]^{\cdot x}&& H^t_\m(R) \ar[r] & 0
}
\]
where the leftmost vertical map is the one induced by the natural projection $R/(x^n) \to R/(x)$. If, by way of contradiction, we assume that $H^t_\m(R)$ has finite length, then the middle map $\cdot x^{n-1}: H^t_\m(R) \to H^t_\m(R)$ is the zero map for $n \gg 0$. The Snake Lemma now implies that $H^{t-1}_\m(R/(x^n)) \to H^{t-1}_\m(R/(x))$ cannot be surjective for $n \gg 0$, contradicting the fact that $x$ is a surjective element.
\end{proof}
\subsection{Flat base change} In this subsection we consider flat base change. Our first result deals with flat extension of the ambient regular local ring.
\begin{Lemma} \label{Lemma_fullness_flat_extension} Let $(S,\m) \to (T,\n)$ be a flat map of unramified regular local rings and $I \subseteq S$ be an ideal such that $S$ and $S/I$ have the same characteristic. If $S/I$ is cohomologically full, then so is $T/IT$. If $(S,\m) \to (T,\n)$ is faithfully flat and $T/IT$ is cohomologically full, then so is $S/I$.

In particular, if $R$ is a homomorphic image of an unramified regular local ring of the same characteristic, then $R$ is cohomological full implies $R_Q$ is cohomologically full for every prime $Q\subseteq R$.
\end{Lemma}
\begin{proof}
By Remark \ref{Remark_cohomfull} $(1)$ we may assume $S$, $T$ are both complete. For every positive integer $j$, consider the following commutative diagram:
\[\xymatrix{
\Ext_S^j(S/I, S) \ar[rr]^\varphi \ar[d] && H_I^j(S) \ar[d]\\
\Ext_S^j(S/I, S)\otimes T\cong \Ext_T^j(T/IT, T) \ar[rr]^-{\varphi\otimes\id_T} && H_I^j(S)\otimes T\cong H_{IT}^j(T)
}\]
Since $T$ is flat over $S$, it is clear that $\varphi$ is injective implies $\varphi\otimes\id_T$ is injective. In addition, if $T$ is faithfully flat over $S$, then $\varphi\otimes\id_T$ is injective also implies $\varphi$ is injective. The result now follows from $(5) \Rightarrow (1)$ of Proposition \ref{Prop_characterization_cohom_full}.

Finally, the last conclusion of the lemma follows because we can write $R=S/I$ for an unramified regular local ring $(S, \m)$ and a localization of $S$ is obvious flat over $S$ and is still unramified: this is vacuous in equal characteristic, and in mixed characteristic, this follows from the fact that for every prime ideal $P\subseteq S$, we have $P^{(2)}\subseteq \m^2$.
\end{proof}

\begin{Corollary}
\label{Cor_Frobpower_char_p}
Let $(S,\m,k)$ be a regular local ring of characteristic $p>0$, and $I \subseteq S$ be an ideal. Then $S/I$ is cohomologically full if and only if $S/I^{[p^e]}$ is cohomologically full for all $e$.
\end{Corollary}
\begin{proof} This follows immediately from applying Lemma \ref{Lemma_fullness_flat_extension} to the $e$-th iteration of the Frobenius map of $S$, which is a faithfully flat map \cite{Kunz1969}.
\end{proof}

As a consequence, we observe that defining ideals of cohomologically full rings in positive characteristic give precisely the answer to  \cite[Question 6.1]{EisenbudMustataStillman}.

\begin{Corollary}
\label{EMS_Q1}
Let $(S,\m,k)$ be a regular local ring of characteristic $p>0$, and $I \subseteq S$ be an ideal. Then the following are equivalent:
\begin{enumerate}
\item  $S/I$ is cohomologically full.
\item  There is a decreasing sequence of ideals $\{I_e\}$  cofinal with the ordinary powers $\{I^e\}$ such that $H_I^i(S) = \bigcup_e\Ext_S^i(S/I_e, S)$.
\item $H_I^i(S) = \bigcup_e\Ext_S^i(S/I^{[p^e]}, S)$.

\end{enumerate}
\end{Corollary}

\begin{proof}
It is tautological that $(3)\Rightarrow(2)\Rightarrow(1)$.  $(1)\Rightarrow(3)$ follows from Corollary \ref{Cor_Frobpower_char_p} and Proposition \ref{Prop_characterization_cohom_full}.
\end{proof}

We next obtain a stronger version of Corollary \ref{Coroll_Positive_Depth}. This will help us construct examples and counterexamples later in this article. Recall that a ring $R$ satisfies {\it Serre's condition $(S_k)$} if $\depth(R_P) \geq \min\{k,\height(P)\}$ for all $P \in \Spec(R)$.

\begin{Lemma}
\label{Lem_full_S_1}
Let $(S,\m,k)$ be an unramified regular local ring of dimension $n$ and $I\subseteq S$ be an ideal such that $S$ and $S/I$ have the same characteristic. If $S/I$ is a cohomologically full ring, then $S/I$ satisfies Serre's condition $(S_1)$. In particular, it has no embedded associated primes.
\end{Lemma}
\begin{proof}
If $\dim(S/I)=0$ there is nothing to show, so let us assume that $S/I$ has positive dimension. By Corollary \ref{Coroll_Positive_Depth}, we have $\depth(S/I)>0$. The conclusion follows since (under our assumptions) the condition of being cohomologically full localizes by Lemma \ref{Lemma_fullness_flat_extension}.
\end{proof}

We next deal with more general faithfully flat base change. Using the result on deformation, we can prove the following:

\begin{Proposition} \label{Prop_flat_base_full}
Let $(A,\m)\to (B,\n)$ be a flat local extension between local rings that are homomorphic images of unramified regular local rings. Assume that all rings involved have the same characteristic. Suppose $B/\m B$ is Cohen-Macaulay. If $A$ is cohomologically full, then so is $B$.
\end{Proposition}
\begin{proof}
By Remark \ref{Remark_cohomfull} $(1)$ we may assume that $A$ and $B$ are both complete. Let $x_1,\dots,x_t$ be a maximal regular sequence in $B/\m B$, then $x_1,\dots,x_t$ is a regular sequence on $B$ and $A\to B'=B/(x_1,\dots,x_t)$ is still faithfully flat with $B'/\m B'$ Artinian. If we can show $B'$ is cohomologically full, then $B$ will also be cohomologically full by Theorem \ref{Thm_deformation} (note that $B'$ is faithfully flat over $A$ and hence have the same characteristic as $A$ and $B$). Hence we may assume $\dim A=\dim B$.

We can form the following commutative diagram:
\[\xymatrix{
S \ar[r] \ar[d] & T\ar[d]\\
A=S/I \ar[r] & B=T/J
}
\]
where $S$, $T$ are complete and unramified regular local rings that surject onto $A$, $B$ respectively (and have the same characterstic). Since $I=J\cap S$, $I^n\subseteq J^n\cap S\subseteq I$. Hence $J^n\cap S$ are thickenings of $I$. Thus for every $i, n\geq 0$, we have the induced commutative diagram on local cohomology (abuse of notation, $\m$, $\n$ denote the maximal ideals of $S$, $T$ respectively):
\[\xymatrix{
H_\m^i(A)=H_\m^i(S/I) \ar[r]^\alpha & H_\n^i(B)=H_\n^i(T/J)\\
H_\m^i(S/J^n\cap S) \ar[r] \ar@{->>}[u] & H_\n^i(T/J^n)\ar[u]
}
\]
The left vertical map is surjective because $A$ is cohomologically full and the map $\alpha$ is the base change $- \otimes_A B$ (because $\dim A=\dim B$) hence it is surjective up to $T$-span. Therefore the $T$-span of the image of $H_\m^i(S/J^n\cap S)$ inside $H_\n^i(T/J)$ is equal to $H_\n^i(T/J)$. This implies $H_\n^i(T/J^n)\to H_\n^i(T/J)$ is surjective for every $i,n$ because its image is already a $T$-module. Therefore $B=T/J$ is cohomologically full by Proposition \ref{Prop_characterization_cohom_full} $(6)\Rightarrow(1)$.
\end{proof}

It is also quite natural to ask whether under a flat local extension $A\to B$, $B$ being cohomologically full implies $A$ being cohomologically full? This question has a positive answer in characteristic $p>0$.

\begin{Proposition} \label{Prop_converse_base_change_char_p}
Let $(A,\m)\to (B,\n)$ be flat local extension of local rings of characteristic $p>0$ that are homomorphic images of unramified regular local rings of characteristic $p>0$. If $B$ is cohomologically full, then so is $A$.
\end{Proposition}
\begin{proof}
First of all we can pick a minimal prime $Q$ of $\m B$, the map $(A,\m)\to (B_Q, QB_Q)$ is still faithfully flat and $B_Q$ is cohomologically full by Lemma \ref{Lemma_fullness_flat_extension}. Thus without loss of generality we may assume $\dim A=\dim B$. Since we are in characteristic $p>0$, cohomologically full is the same as F-full by Corollary \ref{Cor_Cohfull=Ffull}. Thus it is enough to prove that, in this case, $A$ is F-full if and only if $B$ is F-full. We consider the following commutative diagram:
\[\xymatrix{
\mathcal{F}_A(H_\m^i(A))\ar[r]\ar[d] & B\otimes \mathcal{F}_A(H_\m^i(A)) \ar[r]^\cong \ar[d] & \mathcal{F}_B(H_\n^i(B)) \ar[d]\\
H_\m^i(A) \ar[r] & B\otimes H_\m^i(A) \ar[r]^\cong &  H_\n^i(B).
}
\]
Since $B$ is faithfully flat over $A$, the left vertical map is surjective if and only if the right vertical map is surjective. This finishes the proof.
\end{proof}

At the moment, we do not know whether Proposition \ref{Prop_converse_base_change_char_p} holds true in general. We propose it here as a question.

\begin{Question} \label{Quest_converse_base_change}
Let $(A,\m)\to (B,\n)$ be flat local extension of local rings. If $B$ is cohomologically full, is $A$ cohomologically full?
\end{Question}

\begin{Remark} \label{Remark_deformation_char0}
To answer Question \ref{Quest_converse_base_change} in equal characteristic $0$, we can reduce to the case that $A$ is cohomologically full on the punctured spectrum (by Lemma \ref{Lemma_fullness_flat_extension}) and $B$ is a finite free $A$-module (via a reduction argument similar to \cite[Proof of Lemma 5.1]{MaLech}).
\end{Remark}

\subsection{Thickenings} In this subsection we study cohomologically full rings under thickenings. In general, if a thickening of $S/I$ is cohomologically full, then $S/I$ needs not be cohomologically full. We give an example.

\begin{Example}
Let $S=k\ps{x,y,z}$ and let $I=(x^4, x^3y, x^2y^2z, xy^3, y^4)$. Since $x^2y^2 \in I:(x,y,z)$, we have that $\depth S/I=0$. However, $\dim S/I=1$, so $S/I$ is not cohomologically full by Corollary \ref{Coroll_Positive_Depth}. On the other hand, one can check that $I^2=(x,y)^8S$, and thus $S/I^2$ is a one dimensional Cohen-Macaulay local ring. It follows that $S/I^2$ is cohomologically full.
\end{Example}

Another natural question arising from our definition is the following: if $R$ is cohomologically full, then is $R_{\red}$ cohomologically full? This question also has negative answer.
\begin{Example}
Let $k$ be a separably closed field of characteristic $p>0$, and let $R=k\ps{s^4,s^3t,st^3,t^4}$. It was shown by Hartshorne \cite{Hartshorne_CI} that $R$ is a set-theoretic complete intersection, hence there is a thickening $T$ of $R$ that is Cohen-Macaulay, hence cohomologically full. However, $R=T_{\red}$ is not cohomologically full: $R$ is a two-dimensional complete local domain with separably closed residue field that is not Cohen-Macaulay, so it cannot be cohomologically full by Corollary \ref{coh_full_dim2}.
\end{Example}

One could also ask whether the opposite implication holds, that is, whether $R_{\red}$ being cohomologically full implies $R$ is cohomologically full. This implication, however, is even less likely than its converse, because it is very easy to find examples of rings $R$ that are not cohomologically full, but $R_{\red}$ is even Cohen-Macaulay. One could then hope that, assuming $R=S/I$ and $R_{\red}=S/\sqrt{I}$ have the same depth, $R$ is cohomologically full if $R_{\red}=S/\sqrt{I}$ is cohomologically full. However, we next give an example which shows that this is not true, even if we consider a monomial thickening of a square-free monomial ideal.

\begin{Example} \label{example_monomial_thickening}
Let $S=k\ps{x,y,z,w}$ and let $I=(x,y)\cap (z,w) \cap (x^2,z^2,w)$. It follows from Lemma \ref{Lem_full_S_1} that $R=S/I$ is not cohomologically full, because it is not $(S_1)$. Moreover we have $\depth(R) = \depth(R_{\red}) = 1$. However, we have $R_{\red} = k\ps{x,y,z,w}/(x,y) \cap (z,w)$ is cohomologically full by Remark \ref{Remark_cohomfull} $(5)$.
\end{Example}

Our main result in this subsection shows a good control of thickenings by a nonzerodivisor:
\begin{Theorem} \label{Thm_full_powers}
Let $(R,\m,k)$ be a local ring that is a homomorphic image of an unramified regular local ring $S$. Let $x$ be a nonzerodivisor on $R$ such that $S$, $R$, and $R/(x)$ have the same characteristic. Consider the following conditions:
\begin{enumerate}
\item $R/(x)$ is cohomologically full
\item $R/(x^n)$ is cohomologically full for some $n\geq 1$
\item $R/(x^n)$ is cohomologically full for all $n\geq 1$
\end{enumerate}
Then we have $(3)\Rightarrow(2)\Rightarrow(1)$. Write $R=S/I$, we have $(1)\Rightarrow(3)$ if there exists a system of ideals $\{I_e\}$, cofinal with the ordinary powers $\{I^e\}$, such that $x$ is a surjective element for $S/I_e$ for all $n$. In particular, this holds if $S$ has characteristic $p>0$.
\end{Theorem}
\begin{proof}
First of all $(3)\Rightarrow(2)$ is obvious. Next we show $(2)\Rightarrow(1)$, we will actually prove the following:
\begin{Claim*}
If $R/(x^n)$ is cohomologically full, then $R/(x^m)$ is cohomologically full for all $m\leq n$.
\end{Claim*}
\begin{proof}[Proof of Claim]
Since $R/(x^n)$ is cohomologically full, we know that $x^n$ is a surjective element. It follows from \cite[Proposition 3.3]{MaQuy} that $x$ is a surjective element and $H_\m^i(R/(x^n))\to H_\m^i(R/(x^m))$ is surjective for all $m\leq n$. But $R/(x^n)$ is cohomologically full implies $H_\m^i(T)\to H_\m^i(R/(x^n))$ is surjective for all thickenings $T$ of $R/(x^n)$. Therefore $H_\m^i(T)\to H_\m^i(R/(x^m))$ is surjective for all thickenings $T$ that factors through $R/(x^n)$, this implies $R/(x^m)$ is also cohomologically full (for example, use Proposition \ref{Prop_characterization_cohom_full} $(1)\Leftrightarrow(6)$).
\end{proof}

Finally we prove $(1)\Rightarrow(3)$ if there exists a system of ideals $\{I_e\}$, cofinal with the regular powers $\{I^e\}$, such that $x$ is a surjective element for $S/I_e$, for all $e$. By the above Claim, it is sufficient to prove that $R/(x)$ is cohomologically full implies $R/(x^2)$ is cohomologically full: then iterate this we know that $R/(x^{2^n})$ is cohomologically full for all $n$ and the Claim will establish $(3)$. Note that the system $\{I_n+(x^n)\}$ is a thickening of $I+(x^2)$ that is cofinal with the regular powers $\{(I+(x^2))^n\}$, therefore it is sufficient to prove that $H^i_\m(S/I_n+(x^n)) \to H^i_\m(S/I+(x^2))$ is surjective for all $n$. This map factors as the composition $H^i_\m(S/I_n+(x^n)) \to H^i_\m(S/I_n+(x^2)) \to H^i_\m(S/I+(x^2))$. The first map is surjective because, by assumption, $x$ is a surjective element for $S/I_n$. Thus, it is enough to show that $H^i_\m(S/I_n+(x^2)) \to H^i_\m(S/I+(x^2))$ is surjective. Because $x$ is a surjective element for $S/I$ and $S/I_n$, the following diagram has exact rows:
\[\xymatrix{
0 \ar[r] & H_\m^i(S/I_n+(x))  \ar[r] \ar[d] & H_\m^i(S/I_n+(x^2)) \ar[r] \ar[d]  & H_\m^i(S/I_n+(x))  \ar[r] \ar[d] & 0\\
0 \ar[r] & H_\m^i(S/I+(x)) \ar[r]  & H_\m^i(S/I+(x^2)) \ar[r]& H_\m^i(S/I+(x)) \ar[r]& 0
}
\]
The first and third columns are surjections because $S/I+(x)$ is cohomologically full by assumption. Therefore the map in the middle is surjective as well, as desired. The final claim follows from the fact that, in characteristic $p>0$, the sequence of ideals $\{I^{[p^e]}\}$ satisfies the desired condition, since if $S/I+(x)$ is cohomologically full, then so is $S/I^{[p^e]} + (x^{p^e})$ for all $e$, by Corollary \ref{Cor_Frobpower_char_p}. The Claim then implies that $S/I^{[p^e]}+(x)$ is cohomologically full as well, and therefore $x$ is a surjective element for $S/I^{[p^e]}$.
\end{proof}

In general, $R$ is cohomologically full and $x$ is a surjective element does not imply $R/(x)$ is cohomologically full. We give an example.

\begin{Example} Let $S=k\ps{x,y,z}$ and $R=S/J$, where $J=(xy,xz)$. Then $R$ is cohomologically full by Remark \ref{Remark_cohomfull} $(5)$. One can check that $\Ass\Ext_S^2(S/J, S)=\Ass H_J^2(S)=\{(y,z)\}$. Pick a nonzerodivisor $r\notin (y,z)$, so that multiplication by $r$ on $\Ext_S^2(S/J,S)$ is injective. Applying Matlis duality, we deduce that multiplication by $r$ is surjective on $H_\m^1(R)$, and thus $r$ is a surjective element of $R$. However, because $\dim R/(r)=1$ and $\depth R/(r)=0$, the ring $R/(r)$ is not cohomologically full by Corollary \ref{Coroll_Positive_Depth}.
\end{Example}

\subsection{Reduction to characteristic $p>0$} For more details about the process of reduction to characteristic $p$ we refer to \cite[Sections 2.1 and 2.3]{HHChar0} and \cite[Setup 5.1]{MaSchwedeShimomoto}. Here we use the same notation as in \cite[Setup 5.1]{MaSchwedeShimomoto}.

Suppose $(R,\m)$ is a local ring essentially of finite type over $\CC$ (or any other field of characteristic zero). Then $R$ is the homomorphic image of $T_P$, where $T=\CC[x_1,\ldots,x_t]$ and $P$ is a prime ideal of $T$, so that $R \cong (T/J)_P$ for some ideal $J \subseteq T$. We pick a finitely generated regular $\Z$-algebra $A \subseteq \CC$, in a way that all the coefficients of the generators of $P$ and of $J$ belong to $A$. Form $T_A=A[x_1,\ldots,x_t]$, and let $P_A = P \cap A$, and $J_A = J \cap A$. Observe that $P_A \otimes_A \CC \cong P$, and $J_A \otimes_A \CC \cong J$. By generic flatness, we may replace $A$ by the localization $A_a$, for some non-zero $a \in A$, and assume that $R_A \cong (T_A/J_A)_{P_A}$ is flat over $A$. If $\n$ is a maximal ideal of $A$, and we let $\kappa = A/\n$, then we can consider $R_\kappa = R_A \otimes_A \kappa$. This is now a local ring, essentially of finite type over a field of characteristic $p>0$. In what follows, by abusing notation, we will say that $\kappa = A/\n$ belongs to a set $S \subseteq \Max\Spec(A)$ to mean that $\n$ belongs to $S$. We consider the following problem.

\begin{Question} \label{Quest_reduction_char_p}
Let $(R,\m)$ be a local ring essentially of finite type over $\CC$. With the notation introduced above, is it true that $R$ is cohomologically full if and only if $R_\kappa$ is cohomologically full for all $\kappa$ in a dense subset $S$ of $\Max\Spec(A)$?
\end{Question}

\begin{Remark}
One could potentially hope that $R$ cohomologically full implies that $R_\kappa$ is cohomologically full for all $\kappa$ in a dense {\it open} subset of $\Max\Spec(A)$. However, this is not the case. 
In fact, the ring $R=\CC[x,y,z]/(x^3+y^3+z^3)\# \CC[s,t]$ is Du Bois and hence cohomologically full by \cite[Lemma 3.3]{MaSchwedeShimomoto}. However, setting $A=\Z$, one can check that $R_p = R \otimes_\Z \Z/(p)$ is cohomologically full if and only if $\Z/(p)[x,y,x]/(x^3+y^3+z^3)$ is F-pure by \cite[Example 2.5]{MaQuy}. It is well-known that this happens if and only if $p \equiv 1$ mod $3$.
\end{Remark}

We can answer Question \ref{Quest_reduction_char_p} in one direction:

\begin{Proposition} \label{Prop_red_char_p} Let $(R,\m)$ be a local ring essentially of finite type over $\CC$, and let $A$, $T_A$, $P_A$ and $J_A$ be as above. If $R_\kappa$ is cohomologically full for all $\kappa$ in a dense subset $S$ of $\Max\Spec(A)$, then $R$ is cohomologically full.
\end{Proposition}
\begin{proof}
Fix an integer $n \geq 1$, and consider the natural map
\[
\ds \Ext^i_{T_A}(R_A,T_A) \cong \Ext^i_{T_A}(T_A/J_A,T_A) \to \Ext^i_{T_A}(T_A/J_A^n,T_A).
\]
Denote by $K$ its kernel. After inverting a non-zero element $a \in A$, by generic freeness we can assume that $K_a$ is free over $A_a$. Furthermore, inverting more elements if needed, we may assume that for all $\kappa = A/\n$ with $a \notin \n$ we have
\[
\ds \Ext^i_{T_A}(T_A/J_A,T_A)_a \otimes_{A_a} \kappa \cong \Ext^i_{T_\kappa}(T_\kappa/J_\kappa,T_\kappa),
\]
and also
\[
\ds \Ext^i_{T_A}(T_A/J^n_A,T_A)_a \otimes_{A_a} \kappa \cong \Ext^i_{T_\kappa}(T_\kappa/J^n_\kappa,T_\kappa).
\]
See \cite[Theorem 2.3.5]{HHChar0} and the preceding discussion for more details on this. It follows that $K_a \otimes_{A_a} \kappa$ is the kernel of the map $\Ext^i_{T_\kappa}(T_\kappa/J_\kappa,T_\kappa) \to \Ext^i_{T_\kappa}(T_\kappa/J^n_\kappa,T_\kappa)$. Since $S \subseteq \Max\Spec(A)$ is dense, we can find a maximal ideal $\n' \in S$ such that $a \notin \n'$. Let $\kappa' = A/\n'$, so that $R_{\kappa'}$ is cohomologically full by assumption. Then we have that $K_a \otimes_{A_a} \kappa' = 0$, since the map $\Ext^i_{T_{\kappa'}}(T_{\kappa'}/J_{\kappa'},T_{\kappa'}) \to \Ext^i_{T_{\kappa'}}(T_{\kappa'}/J^n_{\kappa'},T_{\kappa'})$ is an injection by assumption. However, this forces $K_a$ to be zero by freeness. In particular, $K$ is a torsion $A$-module, and thus $K \otimes_A \CC = 0$. Since $\Ext^i_{T_A}(R_A,T_A) \otimes_A \CC \cong \Ext^i_{T}(T/J,T)$ and $\Ext^i_{T_A}(T_A/J_A^n,T_A) \otimes_A \CC \cong \Ext^i_{T}(T/J^n,T)$, we get that $\Ext^i_{T}(T/J,T)\to \Ext^i_{T}(T/J^n,T)$ is injective. This is still true after localizing at $P$, so that
\[
\ds \Ext^i_{T_P}(R,T_P) \cong \Ext^i_{T_P}((T/J)_P,T_P) \to \Ext^i_{T_P}((T/J^n)_P,T_P)
\]
is injective. The statement now follows by taking the direct limit over $n$, and using Proposition \ref{Prop_characterization_cohom_full} $(5)\Leftrightarrow(1)$.
\end{proof}

\begin{Remark}
The converse direction of Proposition \ref{Prop_red_char_p} seems to be related to the weak ordinarity conjecture \cite{MustataSrinivasOrdinaryVarieties}, and we expect it to be difficult to prove. See Conjecture \ref{conj CFd} for further connections between cohomological fullness and the weak ordinarity conjecture.
\end{Remark}

Proposition \ref{Prop_red_char_p} can be used to prove that certain rings are cohomologically full in characteristic zero. The next example also follows from our results in Section 4, see Proposition \ref{Prop_coh_full_char0}. Nonetheless, we believe that checking this example via reduction mod $p>0$ is interesting in its own. 

\begin{Example} Let $R=\CC[x,y,z]/(x^4+y^4+z^4)\# \CC[s,t]$. Then $R_\m$ is cohomologically full (where $\m$ denotes the unique homogeneous maximal ideal of $R$). In fact, we let $A=\Z$, and we consider the rings
\[
\ds R_p =\Z/(p)[x,y,z]/(x^4+y^4+z^4)\# \ \Z/(p)[s,t],
\]
for prime numbers $p \in \Z$. We will show that $(R_p)_{\m_p}$ is F-full for all $p \equiv 1$ mod $4$. Since $\depth(R_p) = 2 \leq \dim(R_p) =3$, it is enough to show that for $p \equiv 1$ mod $4$ the following facts hold true:
\begin{enumerate}
\item The Frobenius map is surjective on $H^2_{\m_p}(R_p)_0$.
\item $H^2_{\m_p}(R_p)$ is generated by $H^2_{\m_p}(R_p)_0$ as an $R_p$-module.
\end{enumerate}
To verify (1) and (2), we compute the second local cohomology module of $R_p$ more explicitly, using the K{\"u}nneth formula for local cohomology \cite[Theorem 4.1.5]{GotoWatanabe_I}. Fix a prime $p$, let $B=\Z/(p)[x,y,z]/(x^4+y^4+z^4)$, with maximal ideal $\m_B = (x,y,z)$, and let $C=\Z/(p)[s,t]$, with maximal ideal $\m_C = (s,t)$. Since $H^i_{\m_B}(B) = H^i_{\m_C}(C)=0$ for all $i \ne 2$, we obtain that
\[
\ds H^2_{\m_p}(R_p) \cong\left( H^2_{\m_B}(B) \ \# \ C\right) \oplus \left(B \ \# \ H^2_{\m_C}(C)\right).
\]
Furthermore, because $H^2_{\m_C}(C)_j = 0$ for all $j >-2$, and $H^2_{\m_B}(B)_j=0$ for all $j >1$, we have
\[
\ds H^2_{\m_p}(R_p) = H^2_{\m_p}(R_p)_0 \oplus H^2_{\m_p}(R_p)_1 \cong \left[H^2_{\m_B}(B)_0 \ \# \ \Z/(p) \right] \oplus \left[ H^2_{\m_B}(B)_1 \ \# \ \Z/(p)[s,t]_1\right].
\]
We can explicitly compute a $\Z/(p)$-basis for $H^2_{\m_B}(B)$:
\[
\ds H^2_{\m_B}(B)_0  = \Z/(p) \lr{ \frac{z^2}{xy},\frac{z^3}{x^2y},\frac{z^3}{xy^2}} \hspace{1cm} \mbox{ and } \hspace{1cm}  H^2_{\m_B}(B)_1 = \Z/(p) \lr{\frac{z^3}{xy}}.
\]
Therefore, as a $\Z/(p)$-vector space, we can write $H^2_\m(R_p)$ as follows:
\[
\ds H^2_\m(R_p) \cong \Z/(p)\lr{ \frac{z^2}{xy} \ \# \ 1,\frac{z^3}{x^2y} \ \# \ 1,\frac{z^3}{xy^2} \ \# \ 1} \oplus \Z/(p)\lr{\frac{z^3}{xy}  \ \# \ s,\frac{z^3}{xy} \ \# \ t}.
\]
We now check that (1) and (2) hold by direct computation:
\begin{enumerate}
\item For $p=4m+1$ we have
\[
\ds \left(\frac{z^2}{xy} \ \# \ 1\right)^p = \frac{z^2 \cdot (z^4)^{2m}}{x^{4m+1} y^{4m+1}} \ \# \ 1 = \frac{z^2 \cdot (x^4+y^4)^{2m}}{x^{4m+1} y^{4m+1}} \ \# \ 1 = {2m \choose m}  \frac{z^2}{xy} \ \# \ 1,
\]
where the last equality follows from the fact that all the terms other than ${2m \choose m}x^{4m}y^{4m}$ that come from expanding $(x^4+y^4)^{2m}$ contain a power of either $x$ or $y$ that exceeds $4m+1$. Furthermore, since $2m<p$, the binomial coefficient ${2m \choose m}$ is a unit. Similarly, one can check that:
\[
\ds \left(\frac{z^3}{x^2y} \ \# \ 1 \right)^p = {3m \choose m} \frac{z^3}{x^2y} \ \# \ 1 \hspace{1cm} \mbox{ and } \hspace{1cm} \left(\frac{z^3}{xy^2} \ \# \ 1 \right)^p = {3m \choose m} \frac{z^3}{xy^2} \ \# \ 1.
\]
Because ${3m \choose m}$ is again a unit, we conclude that the Frobenius map is surjective on $H^2_{\m_p}(R_p)_0$.
\item Given that $H^2_{\m_p}(R_p) = H^2_{\m_p}(R_p)_0 \oplus H^2_{\m_p}(R_p)_1$, it is enough to observe that
\[
\ds \frac{z^3}{xy} \ \# \ s =  (z \ \# \ s) \cdot \left(\frac{z^2}{xy} \ \# 1\right)  \hspace{0.5cm} \mbox{ and } \hspace{0.5cm} \frac{z^3}{xy} \ \# \ t =  (z \ \# \ t) \cdot\left(\frac{z^2}{xy} \ \# 1\right).
\]
This shows that $H^2_{\m_p}(R_p)_1 \subseteq R_p \cdot H^2_{{\m_p}}(R_p)_0$, and concludes the proof.
\end{enumerate}
\end{Example}

\section{Applications: Regularity, Kodaira-type vanishing, and Lyubeznik numbers} \label{Section_Applications}

In this section we give some applications and study further nice properties of cohomologically full rings. In what follows, $(R,\m,k)$ denotes either a local ring, or a standard graded algebra over a field $k$.

\subsection{Bounds on the regularity of cohomologically full rings}\label{regularity} Given a standard graded $k$-algebra $(R,\m,k)$, a non-zero finitely generated $\Z$-graded $R$-module $M$, and an integer $i$, the module $H^i_\m(R)$ is $\Z$-graded, and $H^i_\m(M)_{\gg 0}=0$. The $i$-th $a$-invariant of $M$ is defined as
\[
\ds a_i(M) = \sup\{s \in \Z \mid H^i_\m(M)_s \ne 0\},
\]
where $a_i(M) = -\infty$ if $H^i_\m(M) = 0$. The Castelnuovo-Mumford regularity of $M$ is
\[
\ds \reg(M) = \max\{a_i(M) + i \mid i \in \Z\}.
\]
It follows immediately from the definition that, if $(R,\m,k)$ is a cohomologically full standard graded ring, then $a_i(T) \geq a_i(R)$ for all $i$ and all graded surjections $(T,\n) \twoheadrightarrow (R,\m)$, with $T_{\red}=R_{\red}$. In particular, if $R=S/I$ for some standard graded polynomial ring $S$ over $k$, and $J \subseteq I$ is another ideal with $\sqrt{J}=\sqrt{I}$, then $a_i(S/J) \geq a_i(S/I)$ for all $i$.

We now focus on the positive characteristic case, and we recall the definition of F-thresholds. Given two ideals $\a \subseteq \sqrt{J}$ in a ring $R$ of prime characteristic $p>0$, and given an integer $e>0$, we define
\[
\ds \nu_\a^J(p^e) := \max\{t \in \N \mid \a^t \not\subseteq J^{[p^e]}\}.
\]
The F-threshold of $\a$ with respect to $J$ is defined as
\[
\ds c^J(\a) = \lim_{e \to \infty} \frac{\nu_\a^J(p^e)}{p^e}.
\]
F-thresholds were introduced in \cite{MTW} for regular rings, and then generalized and studied in the singular setup in \cite{HMTW}. The limit above was recently shown to exist in greater generality in \cite{DSNBP}.

We now present an upper bound for the $a$-invariants of an ideal $J$ inside a polynomial ring $S$ over a field of prime characteristic. Comparing this result with \cite[Theorem 5.8]{DSNBP}, we observe that here we obtain a similar type of upper bound, but for all the $a$-invariants. Furthermore, we only need that the quotient $S/J$ is cohomologically full.

\begin{Theorem} \label{Thm_a-inv}
Let $S=k[x_1,\ldots, x_n]$, with $\char(k)=p>0$. Let $\a$ and $J$ be graded ideals of $S$ such that $\a \subseteq \sqrt{J}$, and assume that $S/J$ is cohomologically full. Let $d(\a)$ be the maximal degree of a minimal homogeneous generator of a reduction of $\a$. Then
\[
\ds \max\{a_i(S/J) \mid i \in \Z\} \leq c^J(\a)d(\a)-n.
\]
In particular, $\max\{a_i(S/J) \mid i \in \Z\} \leq \mu(J) d(J)-n$.
\end{Theorem}
\begin{proof}
The final claim follows from setting $\a=J$ in the first inequality, and from the fact that $\nu_J^J(p^e) \leq \mu(J)(p^e-1)$ for all $e$ by the pigeonhole principle, so that $c^J(J) \leq \mu(J)$. 
 To prove that $a_i(S/J) \leq c^J(\a)d(\a)$ observe that, for all integers $e>0$, we have graded containments $\a^{\nu_\a^J(p^e)+1} \subseteq J^{[p^e]}$. Furthermore, since $S/J$ is cohomologically full, so is $S/J^{[p^e]}$ by Corollary \ref{Cor_Frobpower_char_p}. By our previous observations, we then conclude that $a_i(S/\a^{\nu_\a^J(p^e)+1}) \geq a_i(S/J^{[p^e]})$. There exists a constant $f$ such that $\reg(S/\a^{\nu_\a^J(p^e)+1}) = d(\a)(\nu_\a^J(p^e)+1) + f$ for all $e \gg 0$ (for example, see \cite[Theorem 0.1]{EisenbudUlrich}). Finally, it can be checked using graded local duality, and the flatness of Frobenius over $S$, that $a_i(S/J^{[p^e]}) = p^e (a_i(S/J)+n)-n$.  Therefore, for all $e \gg 0$ and all $i \in \Z$, we get
\[
\ds d(\a)(\nu_\a^J(p^e)+1) + f - i = \reg(S/\a^{\nu_\a^J(p^e)+1}) -i \geq a_i(S/\a^{\nu_\a^J(p^e)+1}) \geq p^e(a_i(S/J)+n)-n.
\]
Dividing by $p^e$ and taking the limit as $e \to \infty$ gives the desired inequality.
\end{proof}

It has been observed in \cite[Theorem 7.3]{DSNB_F-thresholds_graded} that graded F-pure rings satisfy strong upper bounds for the projective dimension, in the spirit of Stillman's conjecture. We have extended such results to cohomologically full rings in Proposition~\ref{Prop_depth}. It was proved by Caviglia that a Stillman-type bound for the projective dimension is equivalent to a Stillman-type bound for the Castelnuovo-Mumford regularity. Even if Stillman's conjecture has been settled \cite{AnanyanHochster}, finding sharp upper bounds still proves to be very hard. Using Theorem~\ref{Thm_a-inv}, we obtain a strong upper bound on the Castelnuovo-Mumford regularity of cohomologically full rings, which is sharp in the case of an equi-generated complete intersection.
\begin{Corollary}
Let $S=k[x_1,\ldots,x_n]$, with $\char(k)=p>0$. Let $J$ be a graded ideal of $S$ such that $R=S/I$ is cohomologically full, and let $c=n-\dim(R)$. With the notation of Theorem~\ref{Thm_a-inv}, we have $\reg(R) \leq d(J)\mu(J)-c$.
\end{Corollary}
\begin{proof}
It suffices to observe that $\reg(R) = \max\{a_i(S/J)+i \mid i \in \Z\} \leq \max\{a_i(S/J)\mid  i \in \Z \}+ \dim(R)$, and to use the last statement of Theorem~\ref{Thm_a-inv}.
\end{proof}

In order to obtain similar results for algebras over a field of characteristic zero, we need to be able to keep track of the invariants that appear in Theorem \ref{Thm_a-inv} as we perform the reduction to characteristic $p>0$ process. 

Let $R$ be an algebra of finite type over $\CC$, and assume it is standard graded. Then $R$ is the homomorphic image of a polynomial ring $S=\CC[x_1,\ldots,x_t]$, with $\deg(x_i) = 1$ for all $i$. More specifically, we have that $R\cong S/J$ as graded algebras, for some homogeneous ideal $J \subseteq S$. We can choose a finitely generated regular $\Z$-algebra $A \subseteq \CC$, so that the coefficients of the generators of $J$ lie inside $A$. Consider the graded ring $A[x_1,\ldots,x_t]$, with $\deg(a) =0$ for all $a \in A$, and $\deg(x_i)=1$ for all $i$. Let $J_A = J \cap A$, and let $R_A = S_A/J_A$. Observe that $J_A \otimes_A \CC \cong J$. By generic flatness, we may replace $A$ by the localization $A_a$, for some non-zero $a \in A$, and assume that $R_A$ is flat over $A$. If $\n$ is a maximal ideal of $A$, with $\kappa = A/\n$, we let $S_\kappa = S_A\otimes_A \kappa$, $J_\kappa = J_A \otimes_A \kappa$, and $R_\kappa = R_A \otimes_A \kappa$. Given a subset $U \subseteq \Max\Spec(A)$, abusing notation we will say that a property holds for $\kappa \in U$, where $\kappa =A/\n$, to mean that it holds for $\n \in U$. Finally, we denote by $\delta(J)$ the largest degree of a minimal homogeneous generator of a graded ideal $J$.

\begin{Lemma}
Let $S=\CC[x_1,\ldots,x_t]$ be a standard graded polynomial ring, and let $J \subseteq S$ be a homogeneous ideal. Let $R=S/J$, and let $A$, $S_A$, $J_A$ and $R_A$ be as above. There exists a dense open subset $U \subseteq \Max\Spec(A)$ such that
\[
\ds \mu(J) = \mu(J_\kappa), \hspace{1cm} \delta(J) = \delta(J_\kappa) \hspace{1cm} \mbox{ and } \hspace{1cm} a_i(R) = a_i(R_\kappa)
\]
for all $i \in \Z$ and all $\kappa \in U$.
\end{Lemma}
\begin{proof}
Observe that $\mu(J) = \dim_\CC \left(\Tor_1^S(R,\CC)\right)$, and $\delta(J) = \max\{j \in \Z \mid \Tor_1^S(R,\CC)_j \ne 0\}$. By inverting an element of $A$ we may assume that $\Tor_1^{S_A}(R_A,S_A/(x_1,\ldots,x_t)S_A)$ is free over $A$. In particular, we have that
\[
\ds \rank_A \left(\Tor_1^{S_A}(R_A,S_A/(x_1,\ldots,x_t)S_A)\right) = \dim_\CC \left(\Tor_1^T(R,S/(x_1,\ldots,x_t)S)\right) = \mu(J).
\]
Furthermore, by \cite[Theorem 2.3.5 (e)]{HHChar0}, we have that
\[
\ds \Tor_1^{S_A}(R_A,S_A/(x_1,\ldots,x_t)S_A) \otimes_A \kappa \cong \Tor_1^{S_\kappa}(R_\kappa,S_\kappa/(x_1,\ldots,x_t)S_\kappa) \cong \Tor_1^{S_\kappa}(R_\kappa,\kappa)
\]
for all $\kappa$ in a dense open subset of $\Max\Spec(A)$. Counting vector space dimensions, it follows that $\mu(J_\kappa) = \mu(J)$ for all such $\kappa$. Now consider the inclusion of $A$-modules \[
\ds N:=\Tor_1^{S_A}(R_A,S_A/(x_1,\ldots,x_t)S_A)_{\delta(J)} \subseteq M:=\Tor_1^{S_A}(R_A,S_A/(x_1,\ldots,x_t)S_A).
\]
Observe that $N \otimes_A \CC \cong  \Tor_1^S(R,\CC)_{\delta(J)} \ne 0$. By inverting an element in $A$, if needed, we may assume that the $A$-module map $N \hookrightarrow M$ splits. After tensoring with $\kappa = A/\n$, we therefore still have an inclusion $N/\n N \subseteq M/\n M \cong \Tor_1^{S_\kappa}(R_\kappa,\kappa)$. Since $N$ is a finitely generated $A$-module, it follows by Nakayama's Lemma that $N/\n N \ne 0$. This shows that $\Tor_1^{S_\kappa}(R_\kappa,\kappa)_{\delta(J)} \ne 0$, so that $\delta(J_\kappa) \geq \delta(J)$ for all $\kappa$ in a dense open subset of $\Max\Spec(A)$. The other inequality can be obtained with analogous considerations. The argument for $a$-invariants is similar to that for the maximal degree of a minimal generator. Note that, by graded local duality on $S$, we have that
\[
\ds a_i(R) = \max\{j \in \Z \mid H^i_{(x_1,\ldots,x_t)}(R)_j \ne 0\} = \min\{j \in \Z \mid \Ext^{t-i}_S(R,S(-t))_j \ne 0\}.
\]
By inverting an element of $A$, if needed, we may assume that the inclusion of $A$-modules $N':=\Ext^{t-i}_{S_A}(R_A,S_A(-t))_{a_i(R)} \subseteq \Ext^{t-i}_{S_A}(R_A,S_A(-t))=:M'$ splits. Observe that $N' \otimes_A \CC \cong \Ext^{t-i}_S(R,S(-t))_{a_i(R)} \ne 0$. After tensoring the inclusion $N' \subseteq M'$ with $A/\n$, we obtain an inclusion $N'/\n N' \subseteq M'/\n M' \cong \Ext^{t-i}_{S_\kappa}(R_\kappa,S_\kappa(-t))$, which shows that $\Ext^{t-i}_{S_\kappa}(R_\kappa,S_\kappa(-t))_{a_i(R)} \ne 0$. By local duality on $S_\kappa$ we then have
\[
\ds a_i(R_\kappa) = \min\{j \in \Z \mid \Ext^{t-i}_{S_\kappa}(R_\kappa,S_\kappa(-t))_j \ne 0\} \geq a_i(R).
\]
As before, the other inequality is proved in a similar fashion.
\end{proof}

With the notation introduced above, we say that an algebra $R$ of finite type over $\CC$ is of dense F-full type if $R_\kappa$ is cohomologically full for all $\kappa$ in a dense subset of $\Max\Spec(A)$.

\begin{Corollary}
Let $(S,\m,\CC)$ be a standard graded polynomial ring. Let $J$ be an ideal of $S$ such that $R=S/J$ is of dense F-full type. Then $\max\{a_i(S/J) \mid i \in \Z\} \leq \delta(J)\mu(J)-n$. In particular, if $c=\dim(S)-\dim(R)$, we have $\reg(R) \leq \delta(J)\mu(J) - c$.
\end{Corollary}

\subsection{Kodaira Vanishing type results}\label{Kod}
This subsection is devoted to the study of the following question.
\begin{Question} \label{Quest_Kodaira} Let $S=k[x_1,\dots,x_n]$ be an $n$-dimensional standard graded polynomial ring over a field $k$, and $I \subseteq S$ be a homogeneous ideal. When is $H^j_\m(\Ext^i_S(S/I,S(-n)))_{>0}=0$ for all $i,j$? 
\end{Question}

The vanishing of $H^j_\m(\Ext^i_S(S/I,S(-n)))_{>0}$ when $j>1$ is the local cohomology version of the Kodaira vanishing condition for $X=\Proj S/I$. Question \ref{Quest_Kodaira} has a positive answer when $R$ is F-pure in characteristic $p>0$ or when $X$ is Du Bois in characteristic $0$ (for example, see \cite{DaoMaVarbaroSerreCondition}). We were not aware of any counter-examples to Question \ref{Quest_Kodaira} when $R$ is F-injective in characteristic $p>0$, or when $R$ is cohomologically full. Below we give some partial answers.

We recall that a local ring $(R,\m,k)$ of characteristic $p>0$ is called {\it F-injective} if, for all integers $i$, the map $F:H^i_\m(R) \to H^i_\m(R)$ induced by the Frobenius endomorphism on $R$ is injective.

\begin{Proposition}
Let $S=k[x_1,\dots,x_n]$ be an $n$-dimensional standard graded polynomial ring over a field $k$ of characteristic $p>0$, and $I \subseteq S$ be a homogeneous ideal. Suppose $R=S/I$ is F-injective. Then $H_\m^{d_i}(\Ext^i_S(S/I,S(-n)))_{>0}=0$ for $d_i=\dim(\Ext^i_S(S/I,S(-n)))$. In particular, we have $H^i_\m(\Ext^i_S(S/I,S(-n)))_{>0} = 0$.
\end{Proposition}
\begin{proof}
Given an $R$-module $M$, we denote by $F^e_*M$ the $R$-module with action obtained by restriction of scalars via the $e$-th iterate of the Frobenius map. Since $R$ is F-injective, for all $0 \leq i \leq d$ we have graded injections $H^i_\m(R) \hookrightarrow F^e_*H^i_\m(R)$ of $R$-modules. In turn, these induce graded surjections $F^e_*\Ext^i_S(S/I,S(-n)) \twoheadrightarrow \Ext^i_S(S/I,S(-n))$, by graded local duality. Thus, if $d_i=\dim(\Ext^i_S(S/I,S(-n))) = \dim(F^e_*\Ext^i_S(S/I,S(-n)))$, we get surjections $F^e_*H^{d_i}_\m(\Ext^i_S(S/I,S(-n))) \twoheadrightarrow H^{d_i}_\m(\Ext^i_S(S/I,S(-n)))$ on the top local cohomology modules. Viewing $F^e_*H^{d_i}_\m(\Ext^i_S(S/I,S(-n)))$ as a $p^{-e}$-graded module, the above surjection implies that if $H^{d_i}_\m(\Ext^i_S(S/I,S(-n)))_{t}\neq0$, then $H_\m^{d_i}(\Ext^i_S(S/I,S(-n)))_{p^et}\neq0$ for all $e>0$. This clearly implies $H_\m^{d_i}(\Ext^i_S(S/I,S(-n)))_{>0}=0$, because $H_\m^{d_i}(\Ext^i_S(S/I,S(-n)))_{\gg0}=0$. The last claim follows from the fact that $\dim(\Ext^i_S(S/I,S(-n))) \leq i$.
\end{proof}

Recall that a local ring $(R,\m,k)$ is said to have {\it finite local cohomology}, if $H^i_\m(R)$ has finite length for $i \ne \dim(R)$. When $\widehat{R}$ is equidimensional, $R$ has finite local cohomology if and only if $R_P$ is Cohen-Macaulay for all $P\ne \m$. We next show that Question \ref{Quest_Kodaira} has a positive answer when $S/I$ is cohomologically full and has finite local cohomology. Our proof uses the theory of Eulerian graded $\D$-modules \cite{MaZhang_Eulerian}.

\begin{Theorem} \label{thm_Kodaira}
Let $S=k[x_1,\dots,x_n]$ be an $n$-dimensional standard graded polynomial ring over a field $k$, and $I \subseteq S$ be a homogeneous ideal. Suppose $R=S/I$ is cohomologically full. Then $H_\m^{0}(\Ext^i_S(S/I,S(-n)))_{>0}=0$ for every $i$. In particular, if $R$ has finite length cohomology, then $H_\m^j(R)_{<0}=0$ for all $j<d$.
\end{Theorem}
\begin{proof}
By the graded version of Proposition \ref{Prop_characterization_cohom_full}, we obtain that the map $\Ext^{i}_S(S/I,S) \to H_I^{i}(S)$ is a degree preserving inclusion. Furthermore, since $H^i_\m(S/I)$ has finite length, so does $\Ext^{n-i}_S(S/I,S)$. Applying the functor $H^0_\m(-)$ to the above inclusion we see that
\[
\ds H^0_\m(\Ext^{i}_S(S/I,S)) = \Ext^{i}_S(S/I,S) \to H^0_\m(H_I^{i}(S))
\]
is still an inclusion. Since $H^0_\m(H_I^{i}(S))$ is an Artinian Eulerian graded $\D$-module, it must be isomorphic to a finite direct sum of copies of $^*E(n)$, where $^*E$ is the graded injective hull of $k$ of $S$ \cite[Theorem 1.2]{MaZhang_Eulerian}. In particular, the module $H_\m^0(\Ext^{i}_S(S/I,S))$ only lives in degrees $\leq -n$ and thus $H_\m^{0}(\Ext^i_S(S/I,S(-n)))_{>0}=0$.

Finally, when $R$ has finite length cohomology, by graded local duality, $H^j_\m(R)_{<0}$ is graded dual of $\Ext^i_S(S/I,S(-n))_{>0}\cong H_\m^{0}(\Ext^i_S(S/I,S(-n)))_{>0}$, since $H_\m^j(R)$ has finite length.
\end{proof}


\subsection{Relations with (quasi-)Buchsbaum rings and Lyubeznik numbers}\label{Lyu}
\begin{Definition} Let $(R,\m,k)$ be a $d$-dimensional ring that is a homomorphic image of a regular ring $S$, either local or standard graded, of the same characteristic as $R$. We say that $R$ is quasi-Buchsbaum if, for all $i \ne d$, the module $H^i_\m(R)$ is a $k$-vector space. We say that $R$ is Buchsbaum if the natural map $\Ext^i_S(k,R) \to H^i_\m(R)$ is surjective for all $i \ne d$.
\end{Definition}

The one given here is not the most common definition of Buchsbaum ring, but it is equivalent by \cite[Theorem 1]{SV}. Note that Buchsbaum rings are quasi-Buchsbaum, since a surjection $\Ext^i_S(k,R) \to H^i_\m(R)$ implies that $H^i_\m(R)$ is killed by $\m$.

\begin{Proposition} Let $(R,\m,k)$ be a local ring of characteristic $p>0$ with finite local cohomology. If $R$ is F-injective and $k$ is perfect, then $R$ is cohomologically full and Buchsbaum. If $R$ is quasi-Buchsbaum and cohomologically full, then Frobenius is injective on $H^i_\m(R)$ for all $i \ne \dim(R)$.
\end{Proposition}
\begin{proof}
If $R$ is F-injective and has finite local cohomology, then $R$ is Buchsbaum by \cite[Corollary 1.3]{MaFinjectivityBuchsbaum}. In particular, the local cohomology module $H^i_\m(R)$ is a vector space for all $i \ne d$. Let $v_1,\ldots,v_t$ form a basis of $H^i_\m(R)$. By assumption, the Frobenius map $F: H^i_\m(R) \to H^i_\m(R)$ is injective, hence $F(v_1),\ldots,F(v_t)$ are $k^p$-linearly independent.  Since $k$ is perfect, and $H^i_\m(R)$ is a $k$-vector space, the $R$-span of $F(v_1),\ldots,F(v_t)$ must be $H^i_\m(R)$. This proves $R$ is F-full and hence cohomologically full by Corollary \ref{Cor_Cohfull=Ffull}.

For the second statement, assume that $R$ is quasi-Buchsbaum and cohomologically full. Let $v_1,\ldots,v_t$ be a $k$-basis of the vector space $H^i_\m(R)$, for $i \ne \dim(R)$. Assume that $F(\sum_i r_iv_i)=0$ for some $r_i \in R$, not all belonging to $\m$. It follows that $\sum_i r_i^pF(v_i) = 0$, which shows that $F(v_1),\ldots,F(v_t)$ are linearly dependent over $k$. In particular, the $R$-span of $F(v_1),\ldots,F(v_t)$, which is the $k$-span of $F(v_1),\ldots,F(v_t)$, cannot be the whole $H^i_\m(R)$. This shows $R$ is not F-full and hence not cohomologically full by Corollary \ref{Cor_Cohfull=Ffull}, which is a contradiction.
\end{proof}

In \cite[Example 3.5]{MaSchwedeShimomoto}, it is shown that the assumption that $k$ is perfect in the above cannot be removed. On the other hand,  when $\char(k) \equiv 2$ modulo $3$, $R=k[x,y,z]/(x^3+y^3+z^3) \# k[s,t]$ is an example of a ring that is Buchsbaum \cite[Theorem A]{Miyazaki_Segre}, but for which Frobenius is not injective on $H^2_\m(R)$. In particular, $R$ is not cohomologically full.

We now explore some relations between these dimensions and some Lyubeznik numbers. First, we recall the definition of Lyubeznik numbers.
\begin{Definition} Let $(S,\m,k)$ be an unramified regular ring that is either local or standard graded. Let $I$ be an ideal of $S$, homogeneous in the latter case, and set $R=S/I$. Assume that $S$ and $R$ have the same charactersitic. The Lyubeznik number of $R$ with respect to $i,j$ is defined as
\[
\ds \lambda_{i,j} (R) := \dim_k \Ext^i_S \left(k, H^{n-j}_I (S) \right).
\]
\end{Definition}

In these assumptions, the Lyubeznik number $\lambda_{i,j}(R)$ is an invariant of $R$, meaning that it does not depend on the presentation of $R$ as a quotient of a regular local or graded ring \cite{LyuDMod, NunezBetancourtWitt_LyubeznikMixedChar}. The following result generalizes \cite[Theorem 5.3]{DeStefaniGrifoNunezBetancourt_LyubeznikNumbers}.

\begin{Proposition} \label{Prop-Lyubeznik} Let $(R,\m,k)$ be a standard graded cohomologically full $k$-algebra of characteristic $p >0$. If $R$ has finite local cohomology, then for all $j<d$
\[
\ds \lambda_{0,j}(R) = \dim_k([H_\m^j(R)]_0).
\]
\end{Proposition}
\begin{proof}
Let $\ell$ be a field extension of $k$, and set $\overline{R} = R \otimes_k \ell$. Since $R \to \overline{R}$ is flat with closed fiber a field, we have $\lambda_{0,j}(R) = \lambda(\overline{R})$. Moreover, there is a graded isomorphism $H^i_\m(R) \otimes_R \overline{R} \cong H^i_{\m\overline{R}}(\overline{R}) = H^i_{\overline{\m}}(\overline{R})$. Finally, if $R$ is cohomologically full, then so is $\overline{R}$ by Proposition \ref{Prop_flat_base_full}. Therefore, we may assume that $k$ is perfect. Write $R=S/I$, where $S$ is an $n$ dimensional polynomial ring over $k$. It is proved in \cite{Zhang_Lyubeznik_Projective} that $\lambda_{0,j}(R)$ is the $k$-vector space dimension of the F-stable part of $\Ext^n_S(\Ext^{n-j}_S(S/I,S),S)_0$:
\[
\ds \lambda_{0,j}(R) = \dim_k \left[\bigcap_e F^e\left(\Ext^n_S(\Ext^{n-j}_S(S/I,S),S)_0\right)\right].
\]
Let $\ck{(-)}$ denote the graded Matlis dual of a module. By graded local duality, we have graded isomorphisms
\begin{align*}
\ds \Ext^n_S(\Ext^{n-j}_S(S/I,S),S) & \cong \ck{\left[H^0_\m(\Ext^{n-j}_S(S/I,S(-n)))\right]} \\
&= \ck{\left[\Ext^{n-j}_S(S/I,S(-n))\right]} \cong H^j_\m(S/I),
\end{align*}
where we used that $\lambda(\Ext^{n-j}_S(S/I,S)) < \infty$. By \cite[Proposition 5.7 and Lemma 7.3]{Zhang_Lyubeznik_Projective}, the Lyubeznik number $\lambda_{0,j}(R)$ is then the dimension of the F-stable part of $H^j_\m(R)_0$, where the Frobenius action is the one induced by the natural Frobenius action on $R$. By Lemma \ref{Lemma_Span_Degree_zero}, the module $H^j_\m(R)$ is the $R$-span of $F(H^j_\m(R)_0)$, showing that the entire vector space $H^j_\m(R)_0$ is F-stable.
\end{proof}
We end this subsection showing that, under some assumptions, if $R/(x)$ is cohomologically full then $R$ and $R/(x)$ share essentially the same Lyubeznik table. This argument was shown to us by Luis N{\'u}{\~n}ez-Betancourt and Ilya Smirnov.
\begin{Proposition} \label{Prop_Lyubeznik R/x}
Let $(S,\m,k)$ be an unramified regular ring that is either local or standard graded. Let $I$ be an ideal of $S$, homogeneous in the latter case, and set $R=S/I$. Let $x \in \m$ be an element whose image in $R$ is a nonzerodivisor. Suppose $R/(x)$ is cohomologically full and that $S$, $R$, and $R/(x)$ have the same characteristic. In addition, assume that there exists a sequence of ideals $\{I_e\}$, cofinal with $\{I^e\}$, such that $x$ is a surjective element for $S/I_e$ for all $e$ (note these assumptions are unnecessary when $S$ has characteristic $p>0$). Then
\[
\ds \lambda_{i,j}(R) = \lambda_{i-1,j-1}(R/(x))
\]
\end{Proposition}
\begin{proof}
Let $n=\dim(S)$. Since $x^e$ is a surjective element for $S/I_e$ for all $e$, for all $j \in \N$ we have a commutative diagram
\[
\xymatrix{
0 \ar[r] & \Ext^{n-j}_S(R,S) \ar[d] \ar[r]^-x & \Ext^{n-j}_S(R,S) \ar[r] \ar[d]^-x & \Ext^{n-(j-1)}_S(R/(x),S) \ar[r] \ar[d] & 0 \\
0 \ar[r] & \Ext^{n-j}_S(S/I_2,S) \ar[d] \ar[r]^-{x^2} & \Ext^{n-j}_S(S/I_2,S) \ar[r] \ar[d]^-x & \Ext^{n-(j-1)}_S(S/(I_2,x),S) \ar[r] \ar[d] & 0 \\
&\vdots \ar[d]& \vdots \ar[d]^-x& \vdots\ar[d] &  \\
0 \ar[r] & \Ext^{n-j}_S(S/I_e,S) \ar[d]\ar[r]^-{x^e} & \Ext^{n-j}_S(S/I_e,S)\ar[d]^-x \ar[r] & \Ext^{n-(j-1)}_S(S/(I_e,x),S)\ar[d] \ar[r] & 0 \\
&\vdots & \vdots & \vdots &  \\
}
\]
with exact rows. Taking direct limits over $e$, this gives rise to a short exact sequence
\[
\xymatrix{
0 \ar[r] & H^{n-j}_I(S) \ar[r] & H^{n-j}_I(S)_x \ar[r] & H^{n-(j-1)}_{(I,x)}(S) \ar[r] & 0.
}
\]
Applying $\Hom_S(k,-)$, and noting that $\Ext^i_S(k,H^{n-j}_I(S)_x)=0$ for all $i$, we obtain
\[
\ds \Ext^i_S(k,H^{n-j}_I(S)) \cong \Ext^{i-1}_S(k,H^{n-(j-1)}_{(I,x)}(S)).
\]
for all $i \in \N$. It follows that $\lambda_{i,j}(R) = \lambda_{i-1,j-1}(R/(x))$ for all $i,j$.

When $S$ has characteristic $p>0$, one can take $I_e=I^{[p^e]}$, as shown in the proof of Theorem \ref{Thm_full_powers}.
\end{proof}

\subsection{Thickenings of cohomologically full rings}
In this final subsection we point out that for a cohomologically full ring, many properties descends from its thickening. Let $R = S/I$ where $S$ is a regular (local or graded) ring. Let $\mathcal{X}_i$ be subcategories of mod$(R)$ closed under taking submodules. Let $\mathcal{P}$ be a property defined by the following condition:
\begin{center}
$R$ has property $\mathcal{P}$ if and only if $\Ext^i_S(R,S) \in \mathcal{X}_i$ for all $i$.
\end{center}
For example, $\mathcal{P}$ could be: $R$ has finite local cohomology, $R$ is quasi-Buchsbaum, $a_i(R)\leq 0$ for (some or all) $i$.
\begin{Proposition}
\label{Lemma_Prop_P}
Let $R$, $\mathcal{P}$ and $\mathcal{X}_i$ be as above. Let $T=S/J$ be a thickening of $R$. If $T$ has property $\mathcal{P}$ and $R$ is cohomologically full, then $R$ has property $\mathcal{P}$.
\end{Proposition}
\begin{proof}
Since $R$ is cohomologically full, $H_\m^i(T)$ surjects to $H_\m^i(R)$ and by local duality $\Ext^i_S(R,S)$ is a submodule of $\Ext^i_S(T,S)$. Because $\mathcal{X}_i$ is closed under taking submodules and $T$ has property $\mathcal{P}$, it follows that $\Ext^i_S(R,S) \in \mathcal{X}_i$, and hence $R$ has property $\mathcal{P}$ as well.
\end{proof}

We also have the analog for the Buchsbaum property.
\begin{Proposition}
Let $(R,\m,k)$ be a homomorphic image of a regular ring $S$, either local or standard graded, such that $S$ and $R$ have the same characteristic. Suppose $R$ is cohomologically full and some thickening $T=S/J$ of $R$ is Buchsbaum, then $R$ is Buchsbaum.
\end{Proposition}
\begin{proof}
For all $i<d$ we have a commutative diagram
\[
\xymatrix{
\Ext^i_S(k,T) \ar[d] \ar@{->>}[r] & H^i_\m(T) \ar@{->>}[d] \\
\Ext^i_S(k,R) \ar[r] & H^i_\m(R)
}
\]
where the top horizontal map is onto because $T$ is Buchsbaum, and the right vertical map is onto because $R$ is cohomologically full. It follows that the bottom horizontal map is surjective. Since this holds for all $i<d$, $R$ is Buchsbaum.
\end{proof}

\section{A characterization of cohomological fullness for certain graded algebras and a connection to the weak ordinarity conjecture}

In this section, $R$ will be a standard graded algebra over a field $k$. We give a complete and easily checkable equivalent condition for cohomological fullness under certain conditions. We also discuss and propose a statement that can be seen as a strengthening of the weak ordinarity conjecture  proposed in \cite{MustataSrinivasOrdinaryVarieties}.

\begin{Lemma} \label{Lemma_Span_Degree_zero}
Let $(R,\m,k)$ be a standard graded, $i$-cohomologically full $k$-algebra. If $H^i_\m(R)$  is finitely generated then $H^i_\m(R)_0$ generates $H^i_\m(R)$ (as an $R$-module). Moreover, if $R$ has characteristic $p>0$, then $H^i_\m(R)$ is generated by $F^e(H^i_\m(R)_0)$ for all $e \in \N$ as an $R$-module.
\end{Lemma}
\begin{proof}
Write $R=S/I$ for some $n$-dimensional polynomial ring $S$. By graded local duality, $H^i_\m(R)$ is generated in degree zero if and only if the socle of $\Ext^{n-i}_S(S/I,S)$ is concentrated in degree $-n$. Since $R$ is cohomologically full and $\Ext^{n-i}_S(S/I,S)$ has finite length, we have an injection $\Ext^{n-i}_S(S/I,S) \hookrightarrow H^0_\m(H^{n-i}_I(S))$. However, the latter is an Eulerian $\mathcal{D}$-module, hence isomorphic to a direct sum of $^*E(n)$, where $^*E$ denotes the graded injective hull of $k$ in $S$ \cite[Theorem 1.2]{MaZhang_Eulerian}. In particular, the socle of $\Ext^{n-i}_S(S/I,S)$ is contained in the socle of $\bigoplus ^*E(n)$, which lives in degree $-n$.

If $R$ has characteristic $p>0$, since $R$ is cohomologically full, we have that $H^i_\m(R)_0 = (R[F^e(H^i_\m(R))])_0 = R_0[F^e(H^i_\m(R)_0)]$ for all $e \in \N$. By the first part of the Lemma, we finally obtain $H^i_\m(R) = R[H^i_\m(R)_0] = R[F^e(H^i_\m(R)_0)]$.
\end{proof}

We also have a partial converse of the above result in characteristic $0$:
\begin{Proposition}
\label{Prop_coh_full_char0}
Let $(R,\m, k)$ be a standard graded reduced $k$-algebra with finite local cohomology. Suppose $k$ has characteristic $0$ and $R_P$ is Du Bois for all $P\neq\m$ (e.g., $R$ has an isolated singularity at $\m$) and $H_\m^i(R)_0$ generates $H_\m^i(R)$. Then $R$ is $i$-cohomologically full.
\end{Proposition}
\begin{proof}
Write $R=S/I$ where $S$ is a standard graded polynomial ring over $k$. By the graded analog of Propostion \ref{Prop_characterization_cohom_full}, it is enough to show that $H_\m^i(S/J)\to H_\m^i(R)$ is surjective for every homogeneous $J\subseteq I$ such that $\sqrt{J}=I$. By \cite[Remark 4.6]{MaSchwedeShimomoto}, $H_\m^i(S/J)_0\to H_\m^i(R)_0$ is surjective. Since $H_\m^i(R)_0$ generates $H_\m^i(R)$, this implies $H_\m^i(S/J)\to H_\m^i(R)$ is surjective.
\end{proof}

Combining the previous propositions yields:

\begin{Theorem}\label{H0generates}
Let $(R,\m, k)$ be a standard graded reduced $k$-algebra and assume that $H^i_\m(R)$ is finitely generated.  Suppose $k$ has characteristic $0$ and $R_P$ is Du Bois for all $P\neq\m$. Then $R$ is $i$-cohomologically full if and only if  $H^i_\m(R)_0$ generates $H^i_\m(R)$.
\end{Theorem}

One could hope for an even stronger statement.

\begin{Question}\label{H0generatesQues}
Does Theorem \ref{H0generates} still hold if we only assume that $R_P$ is Du Bois for all $P$ of codimension at most $i-1$?
\end{Question}

The previous results show that the degree zero part of local cohomology modules is particularly relevant in relation to cohomological fullness. In this direction, we pose the following conjecture.

\begin{Conjecture} \label{conj CFd} Let $(R,\m,\CC)$ be a reduced standard graded $\CC$-algebra that has an isolated singularity at $\m$. Let $A$ be a regular subring of $\CC$ such that $R$ is defined over $A$. Then is the Frobenius map surjective on $H^i_{\m_\kappa}(R_\kappa)_0$ for all $i$ for $\kappa=A/\n$ in a dense subset of $\n\in\Max\Spec(A)$?
\end{Conjecture}

We observe that Conjecture \ref{conj CFd} is actually equivalent to the weak ordinarity conjecture \cite{MustataSrinivasOrdinaryVarieties}. First we assume the weak ordinarity conjecture. Replacing $R$ by a high Veronese subring if necessary, we may assume that $H^i_\m(R)_{>0}=0$ for all $i$. This does not affect the degree zero part of $H^i_\m(R)$. Since $R_P$ is regular (in particular Du Bois) for all $P \ne \m$, it follows from \cite[Proposition 4.4]{MaSchwedeShimomoto} that $R$ is Du Bois. It then follows from \cite[Theorem B]{BhattSchwedeTakagi} that the Frobenius map is injective on $H^i_{\m_\kappa}(R_\kappa)$ for $\kappa$ in a dense subset of $\Max\Spec(A)$. In particular the Frobenius map is bijective (and thus surjective) on $H^i_{\m_\kappa}(R_\kappa)_0$ for $\kappa=A/\n$ in a dense subset of $\n\in\Max\Spec(A)$.

Conversely, assume that Conjecture \ref{conj CFd} is true, and let $X$ be a smooth projective variety over $\CC$. Let $L$ be a very ample line bundle on $X$ and let $R = \bigoplus_n H^0(X,L^{\otimes n})$ be the corresponding section ring. Then $(R,\m,\CC)$ is a standard graded normal domain with an isolated singularity at $\m$, and $H^{i+1}_\m(R)_0 = H^{i}(X,\mathcal{O}_X)$ for all $i \geq 1$. By Conjecture \ref{conj CFd}, we conclude that the Frobenius map is surjective (equivalently, bijective) on $H^{i}(X,\mathcal{O}_X)$ for all $i \geq 1$.

Inspired by Conjecture \ref{conj CFd}, we ask the following more general question, in similar spirit to Question \ref{H0generatesQues}.
\begin{Question}  \label{question CFi} Let $(R,\m,\CC)$ be a reduced standard graded $\CC$-algebra. Let $A$ be a regular subring of $\CC$ such that $R$ is defined over $A$. If $R$ satisfies Serre's condition $(R_{i-1})$,  is the Frobenius map surjective on $H^j_{\m_\kappa}(R_\kappa)_0$ for all $j\leq i$ for $\kappa=A/\n$ in a dense subset of $\n\in\Max\Spec(A)$?
\end{Question}

\begin{Remark}
In view of the previous discussion, one could view Question \ref{question CFi} as a strengthening of the weak ordinarity conjecture. It is not hard to see that the question has an affirmative answer when $i=1$. For the Frobenius map is surjective on $H^1_{\m_\kappa}(R_\kappa)_0$ if and only if $\dim_{R_\kappa/\m_\kappa}  H^1_{\m_\kappa}(R_\kappa)_0+1$ equals the number of geometrically connected components of $\Proj R_\kappa$(see \cite{SinghWalther2008})  and the assertion follows from semicontinuity.  Shunsuke Takagi has suggested to us that one could use the Albanese map to settle the case $i=2$, as in \cite{SrinivasTakagi}.
\end{Remark}

\bibliographystyle{alpha}
\bibliography{References}

\end{document}